\documentclass[a4paper,12pt]{article}

\title{\Large Global null controllability  of stochastic semilinear complex Ginzburg-Landau equations}
\author{\sc\normalsize Sen Zhang$^\dagger$, Hang Gao$^\S$, and  Ganghua Yuan$^\ddagger$}
\date{}

\usepackage{amsmath}
\usepackage{amsthm}
\usepackage{amsfonts,amssymb}
\usepackage{geometry}
\usepackage{mathrsfs}
\usepackage{indentfirst}
\usepackage{titlesec}
\usepackage{color}

\newtheorem{theorem}{Theorem}[section]

\newtheorem{lemma}{Lemma}[section]

\newtheorem{remark}{Remark}[section]

\numberwithin{equation}{section}

\newcommand{\dd}[0]{\mathrm{d}}
\newcommand{\EE}[0]{\mathbb{E}}
\newcommand{\dxt}[0]{\dd x\dd t}

\newcommand{\dt}[0]{\dd t}
\newcommand{\dx}[0]{\dd x}
\newcommand{\intt}[0]{\int_0^T}

\newcommand{\rmm}[1]{{\rm #1}}
\newcommand{\barr}[1]{{\overline {#1}}}

\titleformat*{\section}{\normalsize\bfseries\rmfamily}
\titleformat*{\subsection}{\normalsize\bfseries\rmfamily}

\begin{document}
\maketitle
\begin{abstract}
In this paper, we study the null controllability of forward and backward stochastic semilinear complex Ginzburg-Landau equations with global Lipschitz nonlinear terms.
For this purpose, by deriving an improved global Carleman estimates for linear systems,
we obtain the controllability results for the stochastic linear systems with a $L^2$-valued source term.
Based on it, together with a Banach fixed point argument, the desired null controllability of semilinear systems is derived.
\end{abstract}
\footnote{{
\it Key  Words: null controllability, stochastic semilinear Ginzburg-Landau equations, Carleman estimate}


\quad $^\ast$The research of HG was supported in part by National Key R\&D Program of China (No. 2023YFA1009002).
The research of GY was supported in part by NSFC (No. 11771074 and No.12371421) and National Key R\&D Program of China (No. 2021YFA1003400 and No. 2020YFA0714102).

\quad $^\dagger$School of Mathematics and Statistics, Beijing Institute of Technology, Beijing 100081, P.R. China (7520230216@bit.edu.cn).

\quad $^\S$KLAS, School of Mathematics \& Statistics,  Northeast Normal University, Changchun Jilin 130024, P.R. China (hangg@nenu.edu.cn)

\quad $^\ddagger$KLAS, School of Mathematics \& Statistics,  Northeast Normal University, Changchun Jilin 130024, P.R. China (yuangh925@nenu.edu.cn).
}

\section{Introduction and Main Results}
Ginzburg-Landau equation is an important nonlinear mathematical physics equation, which
was proposed initially by Ginzburg and Landau in \cite{Ginzburg1950On} to serve as a phenomenological description of superconductivity.
Now, it is used to describe many other physical phenomena, such as superfluidity, Bose-Einstein condensation, liquid crystals and strings in field theory, nonlinear waves, and second-order phase transition (see e.g., \cite{Aranson2002The, Hohenberg2015An}).
For more physical background and applications, we refer to \cite{Aranson2002The,Garcia-Morales2012The,Milosevic2010The} and the references therein.
In practice, systems are usually affected by uncertainty. When considering the random influence, stochastic complex Ginzburg-Landau equations arise naturally.
In many situations, stochastic
partial differential equations (SPDEs) are more realistic than deterministic ones.
Recently, stochastic complex Ginzburg-Landau equations have been extensively studied (see e.g.,  \cite{Balanzario2020Regularity,Cheng2023Averaging,Fu2017Controllability}).
In this paper, we will study the controllability of stochastic semilinear complex Ginzburg-Landau equations.

Let $G$ be a nonempty bounded domain in $\mathbb{R}^n$ ($n$ is a positive integer) with a boundary $\Gamma$ of class $C^4$ and $T>0$. Put $Q:=(0,T)\times G$ and $\Sigma:=(0,T)\times \Gamma$. Assume $G_{0}$ to be a given nonempty open subset of $G$ and denote by $\chi_{G_{0}}$ the characteristic function of the set $G_{0}$.
For simplicity, we will use the notation $y_{j}:=\frac{\partial y}{\partial x_{j}}$, where $x_{j}$ is the $j$th coordinate of a generic point $x=(x_{1},x_{2},\dots,x_{n})$ in $\mathbb{R}^n$. Similarly, we use the notation $z_{j}$, $v_{j}$, etc. for the partial derivatives of $z$, $v$ with respect to $x_{j}$.
Moreover, for any complex number $c$, we denote by $\barr{c}$, $\rmm{Re}c$, and $\rmm{Im}c$ its complex conjugate, real part, and imaginary part, respectively.
Also, we use $C$ to denote a generic positive constant depending only on $G_{0}$ and $G$,  which may change from line to line.

Let $(\Omega,\mathcal{F},\mathbf{F},\mathbb{P})$ with $\mathbf{F}=\{\mathcal{F}_t\}_{t\geq0}$ be a complete filtered probability space on which a one-dimensional standard Brownian motion $\{B(t)\}_{t\geq0}$ is defined and $\mathbf{F}$ is the natural filtration generated by $B(\cdot)$, augmented by all the $\mathbb{P}$ null sets in $\mathcal{F}$. Let $H$ be a Banach space.
Denote by $L^2_{\mathcal{F}_t}(\Omega;H)$ the space of all $\mathcal{F}_t$-measurable random variables $\zeta$ such that $\EE |\zeta|^2_{H}<\infty$; denote by $L^2_\mathbb{F}(0,T;H)$ the space consisting of all $H$-valued $\mathbf{F}$-adapted processes $X(\cdot)$ such that $\mathbb{E}(|X(\cdot)|^2_{L^2(0,T;H)})<\infty$; by $L^\infty_\mathbb{F}(0,T;H)$ the space consisting of all $H$-valued $\mathbf{F}$-adapted bounded processes $X(\cdot)$; and by $L^2_\mathbb{F}(\Omega;C([0,T];H))$ the space consisting of all $H$-valued $\mathbf{F}$-adapted continuous processes $X(\cdot)$ such that $\mathbb{E}(|X(\cdot)|^2_{C(0,T;H)})<\infty$. Similarly, one can define $L^2_\mathbb{F}(\Omega;C^k([0,T];H))$ for any positive integer $k$. All of these spaces are Banach spaces with canonical norms.

Consider the following controlled forward  stochastic complex  Ginzburg-Landau equation:
\begin{equation}\label{equationy1}
	\left\{
		\begin{aligned}
	&\dd y-(a+\rmm{i} b)\sum_{j,k=1}^n (a^{jk}y_j)_k\dt
	\\
	&\quad\quad\quad\quad\quad =[f(\omega,t,x,y)+\chi_{G_0}h]\dt
	+[g(\omega,t,x,y)+H]\dd B(t) &\textup{in}\ &Q,\\
    &y=0 &\textup{on}\ &\Sigma,\\
    &y(0)=y_0 &\textup{in}\ &G,
            \end{aligned}
    \right.
\end{equation}
and controlled backward stochastic complex Ginzburg-Landau equation:
\begin{equation}\label{equationyb1}
	\left\{
		\begin{aligned}
	&\dd y+(a-\rmm{i} b)\sum_{j,k=1}^n (a^{jk}y_j)_k\dt =[\Upsilon(\omega,t,x,y,Y)+\chi_{G_0}h]\dt
	+Y\dd B(t) &\textup{in}\ &Q,\\
    &y=0 &\textup{on}\ &\Sigma,\\
    &y(T)=y_T &\textup{in}\ &G,
            \end{aligned}
    \right.
\end{equation}
with $a>0$, $b\in\mathbb{R}$.
In (\ref{equationy1}) (resp., (\ref{equationyb1})), $y_{0}$ ($y_{T}$) is a given initial value (resp., terminal value). In both cases, $y$ is the state variable. In (\ref{equationy1}), the control variable consists of the pair $(h,H)$; while in (\ref{equationyb1}), the control variable is only $h$.
Unless otherwise stated, we assume that all functions mentioned
in this paper are complex-valued.
We also assume that the functions $a^{jk}$, $f$, $g$ and $\Upsilon$ satisfy the following assumptions:
\begin{equation}\nonumber
\begin{split}
{\bf (H_{1}).}\ & a^{jk}:\Omega\times [0,T]\times\barr{G}\rightarrow \mathbb{R},\ a^{jk}\in L^2_\mathbb{F}(\Omega;C^1([0,T];W^{2,\infty}(G;\mathbb{R}))),\ j,k=1,2,\dots,n.
\\
{\bf (H_{2}).}\ & \mbox{For any }\epsilon>0,\mbox{ there is a }\rho>0\mbox{ so that }|a^{jk}(\omega,t,x_{1})-a^{jk}(\omega,t,x_{2})|\leq\epsilon
\\
&\mbox{almost surely for any }t\in[0,T]\mbox{ and } x_{1},x_{2}\in \barr{G}\mbox{ satisfying that }|x_{1}-x_{2}|\leq \rho.
\\
{\bf (H_{3}).}\ & a^{jk}=a^{kj}\  \mbox{and there is some constant}\ s_0>0\ \mbox{such that}
\\ &\sum_{j,k=1}^n a^{jk}(\omega,t,x)\xi^{j}\barr{\xi^k}\geq s_0 |\xi|^2\ \mbox{for any}\ (\omega,t,x,\xi )=(\omega,t,x, \xi^{1},\dots,\xi^{n})\in \Omega\times Q\times \mathbb{C}^{n}.
\\
\end{split}
\end{equation}
\begin{equation}\nonumber
\begin{split}
{\bf (H_{4}).}\ &f(\cdot,\cdot,\cdot,y)\ \mbox{and}\ g(\cdot,\cdot,\cdot,y)\  \mbox{are}\ \mathbf{F}\mbox{-adapted},\
L^2\mbox{-valued stochastic processes}
\\
&\mbox{for each}\ y\in L^2(G;\mathbb{C}).
\\
{\bf (H_{5}).}\ &\forall\ (\omega,t,x)\in \Omega\times Q,\ f(\omega,t,x,0)=0.
\\
{\bf (H_{6}).}\ &\exists \ \kappa>0,\ \forall\ (\omega,t,x,a_{1},a_{2})\in \Omega\times Q\times \mathbb{R}^2,\
|f(\omega,t,x,a_{1})-f(\omega,t,x,a_{2})|\leq \kappa |a_{1}-a_{2}|.
\\
{\bf (H_{7}).}\ &\exists \ \kappa_{1}>0,\ \forall\ (\omega,t,x,a_{1},a_{2})\in \Omega\times Q\times \mathbb{R}^2,\ |g(\omega,t,x,a_{1})-g(\omega,t,x,a_{2})|\leq \kappa_{1} |a_{1}-a_{2}|.
\\
{\bf (H_{8}).}\ &\Upsilon(\cdot,\cdot,\cdot,y,Y)\mbox{ are}\ \mathbf{F}\mbox{-adapted},\
L^2\mbox{-valued stochastic processes for each }y,Y\in L^2(G).
\\
{\bf (H_{9}).}\ &\forall\ (\omega,t,x)\in \Omega\times Q,\ \Upsilon(\omega,t,x,0,0)=0.
\\
{\bf (H_{10}).}\ &\exists \ \kappa_{2}>0,\ \forall\ (\omega,t,x,a_{1},a_{2},b_{1},b_{2})\in \Omega\times Q\times \mathbb{R}^4,
\\
&
|\Upsilon(\omega,t,x,a_{1},b_{1})-\Upsilon(\omega,t,x,a_{2},b_{2})|\leq \kappa_{2} (|a_{1}-a_{2}|+|b_{1}-b_{2}|).
\end{split}
\end{equation}
\begin{remark}
Under the assumptions ($H_{1}$)-($H_{7}$), by taking $y_{0}\in L^2_{\mathcal{F}_{0}}(\Omega;L^{2}(G;\mathbb{C}))$ and $(h,H)\in L^2_{\mathbb{F}}(0,T;L^2(G_{0};\mathbb{C}))\times L^2_{\mathbb{F}}(0,T;L^2(G;\mathbb{C}))$, it is known (see e.g. \cite[Theorem 3.24.]{Lu2021Mathematical}) that system (\ref{equationy1}) admits a unique solution
\begin{equation}\nonumber
y\in \mathcal{H}_{T}	
:=
L^2_{\mathbb{F}}(\Omega; C([0,T];L^2(G;\mathbb{C}))) \cap L^2_{\mathbb{F}}(0,T;H_0^1(G;\mathbb{C})).
\end{equation}
Under the assumptions ($H_{1}$)-($H_{3}$) and ($H_{8}$)-($H_{10}$), by taking $y_{T}\in L^2_{\mathcal{F}_{T}}(\Omega;L^{2}(G;\mathbb{C}))$ and $h\in L^2_{\mathbb{F}}(0,T;L^2(G_{0};\mathbb{C}))$, it is known (see e.g. \cite[Theorem 4.11.]{Lu2021Mathematical}) that system (\ref{equationyb1}) admits a unique solution
\begin{equation}\nonumber
y\in \mathcal{H}_{T}\times L^2_{\mathbb{F}}(0,T;L^2(G;\mathbb{C})).
\end{equation}
\end{remark}
The main purpose of this paper is to study the null controllability of the  semilinear stochastic complex Ginzburg-Landau equations (\ref{equationy1}) and
(\ref{equationyb1}), respectively.
System (\ref{equationy1}) (resp., (\ref{equationyb1})) is said to be {\it globally null controllable} if for any $y_{0}\in L^2_{\mathcal{F}_{0}}(\Omega;L^2(G;\mathbb{C}))$ (resp., $y_{T}\in L^2_{\mathcal{F}_{T}}(\Omega;L^2(G;\mathbb{C}))$) , there exists a control $(h,H)\in L^2_{\mathbb{F}}(0,T;L^2(G_{0};\mathbb{C}))\times L^2_{\mathbb{F}}(0,T;L^2(G;\mathbb{C}))$ (resp., $h\in L^2_{\mathbb{F}}(0,T;L^2(G_{0};\mathbb{C}))$) such that the corresponding solution $y$ of (\ref{equationy1}) (resp., (\ref{equationyb1}))   satisfies $y(T,\cdot)=0$ in $G$, $\mathbb{P}$-a.s. (resp., $y(0,\cdot)=0$ in $G$, $\mathbb{P}$-a.s.)

The controllability of deterministic nonlinear partial differential equations (PDEs) has been studied by many authors and the results are relatively rich, such as
nonlinear parabolic equation (see e.g.,  \cite{Barbu2000Exact,Doubova2002On,Emanuilov1995Controllability,Fabre1995Approximate,Fernandez-Cara1995Null,Fernandez-Cara2000Null,Fursikov1996Controllability,LeBalch2020Global}),
nonlinear fourth order  parabolic equation (see e.g., \cite{Kassab2020Null,Zhou2012Observability}),
Ginzburg-Landau equation (see e.g., \cite{Fu2009Null,Rosier2009Null}).
It can be seen from these results that the authors usually use the following strategies to study controllability problems.
First, linearize the nonlinear system and study the controllability of the linearized system.
Then the controllability problem of nonlinear systems is solved by using appropriate fixed point methods, usually Schauder or Kakutani fixed point methods.
At this point, the property of compactness plays a crucial role, for example, the compact embedding result of the Aubin-Lions lemma is used commonly.

Compared with deterministic PDEs, although the controllability problems of SPDEs have not been widely studied, some progress has been made in recent years.
We refer to \cite{Barbu2003Carleman,Fu2017A,Fu2017Controllability,Gao2015Observability,Liao2024Exact,Lu2011Some,Lu2013Exact,Lu2022Null,Tang2009Null,Yu2022Carleman} for some known results in this respect.
Reference \cite{Barbu2003Carleman,Lu2011Some,Tang2009Null} showed the null controllability and approximate controllability for the stochastic parabolic equations.
In \cite{Fu2017A,Fu2017Controllability}, the null controllability for the stochastic complex  Ginzburg-Landau equation was obtained.
In \cite{Lu2013Exact}, the exact controllability for stochastic Schr\"odinger equation was established.
References \cite{Liao2024Exact} and \cite{Yu2022Carleman} showed the exact controllability for a refined stochastic wave equation and a refined stochastic beam equation, respectively.
In \cite{Gao2015Observability} and \cite{Lu2022Null}, the null controllability for the stochastic fourth order parabolic type equations were discussed.

It is worth mentioning that the above works mainly focus on the controllability of linear SPDEs systems, while there is little literature on the controllability of nonlinear stochastic systems.
This is due to the lack of compactness embedding for the function spaces related to SPDEs (see \cite[Remark 2.5]{Tang2009Null} or \cite{Lu2021Mathematical}), which makes some classical strategies in deterministic setting (see e.g., \cite{Fernandez-Cara2000Null}) fail to establish null controllability for semilinear systems at the stochastic level.
As far as we know, the known results in this direction seem to be only \cite{Hernandez-Santamaria2022Statistical,Hernandez-Santamaria2023Global,Zhang2024New}, in which the author established the null controllability of semilinear stochastic parabolic equations.
However, to our best knowledge,  nothing is known about the controllability of nonlinear stochastic complex Ginzburg-Landau equations.
In this paper, we study the null controllability problems for stochastic semilinear complex Ginzburg-Landau equations (\ref{equationy1}) and (\ref{equationyb1}).

The main results in this paper read as follows.
\begin{theorem}\label{control}
Let assumptions $(H_{1})$-$(H_{6})$ be satisfied. Then system $(\ref{equationy1})$ is globally null controllable.
\end{theorem}
\begin{theorem}\label{controlb}
Let assumptions $(H_{1})$-$(H_{3})$ and $(H_{8})$-$(H_{10})$ be satisfied. Then system $(\ref{equationyb1})$ is globally null controllable.
\end{theorem}

Our strategy for proving Theorem \ref{control} is as follows.
To overcome the lack of compactness mentioned above, we borrow some ideas from \cite{Hernandez-Santamaria2023Global,Liu2014Global}.
First, we obtain a new global Carleman estimate for the backward stochastic complex Ginzburg-Landau equation with a $L^2$-valued source by introducing a suitable singular weighted function.
Next, by combining the new Carleman estimate  and Hilbert Uniqueness Method (HUM) introduced in \cite{Lions1988Exact}, we establish the null controllability for a linear system of the form
\begin{equation}\nonumber
	\left\{
		\begin{aligned}
	&\dd y-(a+\rmm{i}b)\sum_{j,k}^{n}(a^{jk}y_{j})_{k}\dt=(F+\chi_{G_0}h)\dt+H\dd B(t) &\textup{in}\ &Q,\\
    &y=0 &\textup{on}\ &\Sigma,\\
    &y(0)=y_0 &\textup{in}\ &G,
            \end{aligned}
    \right.
\end{equation}
where the source $F$ is in some suitable functional space.
At last, we need to prove that a nonlinear mapping $F\rightarrow f(\omega,t,x,y)$ is strictly contractive in a suitable functional space. Through a Banach fixed point method which does not rely on any compactness argument, we can obtain the null controllability for (\ref{equationy1}).
The strategy to prove Theorem \ref{controlb} is very close to that of Theorem \ref{control}, but one major difference can be found.
For this case, it is not necessary to prove a Carleman estimate for the forward stochastic Ginzburg-Landau equation. Actually, it suffices to use the deterministic Carleman inequality and employ the duality method introduced by \cite{Liu2014Global}.

Now, we give some remarks in order.
\begin{remark}\label{remark1}
We claim that the null controllability of equation (\ref{equationy1}) can be reduced to the null controllability of
\begin{equation}\label{equationy2}
	\left\{
		\begin{aligned}
	&\dd y-(a+\rmm{i} b)\sum_{j,k=1}^{n}(a^{jk}y_j)_k\dt=[f(\omega,t,x,y)+\chi_{G_0}h]\dt+H\dd B(t) &\textup{in}\ &Q,\\
    &y=0 &\textup{on}\ &\Sigma,\\
    &y(0)=y_0 &\textup{in}\ &G,
            \end{aligned}
    \right.
\end{equation}
Indeed, assume that one can find two controls $h\in L^2_{\mathbb{F}}(0,T;L^2(G_{0};\mathbb{C}))$ and $H\in  L^2_{\mathbb{F}}(0,T;L^2(G;\\\mathbb{C}))$ such that the corresponding solution $y$ of (\ref{equationy2}) satisfies $y(T,\cdot)=0$ in $G$, $\mathbb{P}$-a.s.
Since the control $H$ is distributed in the whole domain $G$, the state $y$ still satisfies equation (\ref{equationy1}) with the controls $h^{*}=h$ and
\begin{equation}\nonumber
	H^{*}=H-g(\omega,t,x,y)\in L^2_{\mathbb{F}}(0,T;L^2(G;\mathbb{C})),
\end{equation}
which is well defined by $(H_{4})$. Moreover, we still have the
controllability property i.e. $y(T,\cdot)=0$ in $G$, $\mathbb{P}$-a.s. This is why we can drop the Lipschitz condition $(H_{7})$ on $g$ in Theorem \ref{control}.
\end{remark}
\begin{remark}
From the following, it can be found that since $0<T<1$ is assumed, we actually obtain that the system (\ref{equationy1}) and (\ref{equationyb1}) are globally null controllable at any small time.
\end{remark}
\begin{remark}
From Theorem \ref{control} and Theorem \ref{controlb}, we can obtain the null controllability for semilinear stochastic parabolic equations, which implies that the results in this paper include the results in \cite{Hernandez-Santamaria2023Global}.
\end{remark}
\begin{remark}
Theorem \ref{control} requires an extra control $H\in  L^2_{\mathbb{F}}(0,T;L^2(G;\mathbb{C}))$ on the diffusion term.
It would be quite interesting to establish the null controllability for (\ref{equationy1}) by only one control force or the control $H$ acting only on a sub-domain of $G$.
However, this seems difficult for us.
In fact, these problems are still open, even for general linear SPDEs (see e.g., \cite{Fu2017Controllability,Gao2015Observability,Tang2009Null}).
\end{remark}
\begin{remark}
Theorem \ref{control} and Theorem \ref{controlb} provide a partial positive answer to the open question provided in \cite[Remark1.10]{Fu2017Controllability}.
However, the global controllability for system (\ref{equationy1}) and  (\ref{equationyb1}) with more general nonlinearities $f(\cdot)$, $g(\cdot)$ and $\Upsilon(\cdot)$, such as the super-linear nonlinearity considered for deterministic Ginzburg-Landau equation (see e.g., \cite{Fu2009Null}) or parabolic-type PDEs (see e.g., \cite{Fernandez-Cara1995Null,Fernandez-Cara2000Null, Kassab2020Null}), is still an interesting but challenging problem.
\end{remark}
As mentioned above, we prove Theorem \ref{control} and Theorem \ref{controlb} by Carleman estimates.
In the past few years, the Carleman estimates for SPDEs have received much attention.
Although there are numerous results for the global Carleman estimate for deterministic Ginzburg-Landau equation and parabolic-type equation (e.g.,  \cite{Dou2023Global,Fu2009Null,Fursikov1996Controllability}),  the Carleman estimates for the stochastic counterpart are much less.
In this respect, we refer to \cite{Fu2017A,Fu2017Controllability,Liu2014Global,Liu2019Carleman,Lu2011Some,Lu2012Carleman,Tang2009Null,Wu2020Carleman,Yuan2017Conditional} for some known results. In \cite{Fu2017A,Fu2017Controllability}, authors studied Carleman estimats for stochastic complex Ginzburg-Landau equations. The references \cite{Liu2014Global,Liu2019Carleman,Lu2011Some,Lu2012Carleman,Tang2009Null,Wu2020Carleman,Yuan2017Conditional} are devoted to stochastic parabolic equations, where
\cite{Liu2019Carleman,Wu2020Carleman} are concerned with stochastic degenerate parabolic equations. We can also refer to \cite{Hernandez-Santamaria2023Global,Liao2024Exact, Liao2024Stability,Yu2022Carleman,Yuan2021Inverse,Zhang2024New,Zhang2024Unique,Zhang2024Determination} for some
 Carleman estimates of SPDEs and their applications.

The remainder of the paper is organized as follows. In Section 2, we present the proof of Theorem \ref{control}. In particular, we give a new Carleman estimate for the backward stochastic complex Ginzburg-Landau equation. In Section 3, we present the proof of Theorem \ref{controlb}.

\section{Controllability of a semilinear forward stochastic complex Ginzburg-Landau equation}
\subsection{A new Carleman estimate for a backward stochastic complex Ginzburg-Landau equation}
In this part, we need to establish the Carleman estimate for the following backward stochastic complex Ginzburg-Landau equation with a source in drift term:
\begin{equation}\label{equationz}
	\left\{
		\begin{aligned}
	&\dd z+(a-\rmm{i} b)\sum_{j,k=1}^{n}(a^{jk}z_j)_k\dt=\Xi\dt+Z\dd B(t) &\textup{in}\ &Q,\\
    &z=0 &\textup{on}\ &\Sigma,\\
    &z(T)=z_T &\textup{in}\ &G.
            \end{aligned}
    \right.
\end{equation}

At first, we introduce an identity for a stochastic parabolic operator with complex principal parts, which plays a key role in proving the Carleman estimates for complex Ginzburg-Landau equations.
Set
\begin{equation}\nonumber
\mathcal{L}z:=\dd z+(a-\rmm{i} b)\sum_{j,k=1}^{n}(a^{jk}z_j)_k\dt.
\end{equation}
\begin{lemma}[{\cite[Lemma 2.1]{Fu2017Controllability}}]\label{identity}
Suppose that $\hat{\ell} \in C^3(Q;\mathbb{R})$ and $\Phi \in C^1(Q;\mathbb{C})$. Let $\hat{y}$ be an $H^2(G;\mathbb{C})$-valued continuous semimartingle. Set $\hat{\theta}=e^{\hat{\ell}}$ and $v=\hat{\theta} z$. Then for a.e. $(x, t)\in Q$, it holds that	
\begin{equation}\label{identity1}
\begin{split}
 2\rmm{Re}(\hat{\theta}\barr{I_1}{\mathcal L}z)
=&2|I_1|^2\dt
 +\dd M
 +\sum_{k=1}^nV_{k}^k
 +B|v|^2\dt
 +\sum_{j,k=1}^nD^{jk}v_{j}\barr{v}_{k}\dt\\
 &-2\sum_{j=1}^n[\rmm{Re}(aE^j\barr{v}v_j)-\rmm{Im}(bF^jv\barr{v}_j)]\dt
 +\sum_{j,k=1}^naa^{jk}\dd v_j\dd \barr{v}_k\\
 &+(\hat{\ell}_t-aA)\dd v\dd \barr{v}
 -2b\sum_{j,k=1}^na^{jk}\hat{\ell}_k\rmm{Im}(\dd v\dd \barr{v}_j)\\
 &+2\bigg[b\sum_{j,k=1}^n(a^{jk}\hat{\ell}_k)_j\rmm{Im}(v\dd \barr{v})+\rmm{Re}(\barr{\Phi}\barr{v}\dd v)\bigg],
\end{split}
\end{equation}
where
\begin{equation}\label{A}
	\left\{
		\begin{aligned}
	&A=\sum_{j,k=1}^n[a^{jk}\hat{\ell}_{j}\hat{\ell}_{k}-(a^{jk}\hat{\ell}_{j})_{k}],\quad
	\Lambda=\sum_{j,k=1}^n(a^{jk}v_{j})_k+Av,\\
	&I_1=a\Lambda+2\rmm{i}b\sum_{j,k=1}^na^{jk}\hat{\ell}_{j}v_{k}+(\Phi-\hat{\ell}_t)v,
        \end{aligned}
    \right.
\end{equation}
\begin{flalign}
\hspace{5mm}
B=&2(a^2+b^2)\sum_{j,k=1}^n(Aa^{jk}\hat{\ell}_j)_k
   -aA_t
   -2aA\rmm{Re}\Phi
   -2bA\rmm{Im}\Phi
   &\nonumber\\
   &+2\rmm{Re}(\Phi\hat{\ell}_t)
   -2|\Phi|^2
   +\hat{\ell}_{tt},
&\nonumber  \\
M=&aA|v|^2
  +\sum_{j,k=1}^na^{jk}[-av_j\barr{v}_k+2b\hat{\ell}_j\rmm{Im}(\barr{v}_kv)]
  -\hat{\ell}_{t}|v|^2,
&\nonumber  \\
V^k=&2a\sum_{j=1}^na^{jk}\rmm{Re}(v_j\dd \barr{v})
     -2b\sum_{j=1}^na^{jk}\hat{\ell}_{j}\rmm{Im}(v\dd \barr{v})
     -2A(a^2+b^2)\sum_{j=1}^na^{jk}\hat{\ell}_{j}|v|^2\dt
     &\nonumber
\\
     &-2a\sum_{j=1}^na^{jk}\rmm{Re}(\barr{v}_j\Phi v)\dt
     +2b\sum_{j=1}^na^{jk}\rmm{Im}[v_j(\barr{\Phi}-\hat{\ell}_t)\barr{v}]\dt
     &\nonumber \\
  &+2(a^2+b^2)\sum_{j,j',k'=1}^n[a^{jk}a^{j'k'}\hat{\ell}_{j}v_{j'}\barr{v}_{k'}-a^{jk'}a^{j'k}\hat{\ell}_{j}(v_{j'}\barr{v}_{k'}+\barr{v}_{j'}v_{k'})]\dt,
&\nonumber  \\
D^{jk}=&aa^{jk}_{t}
        +2b\rmm{Im}\Phi a^{jk}
        +2a\rmm{Re}\Phi a^{jk}
        &\nonumber \\
        &+2(a^2+b^2)\sum_{j',k'=1}^n[a^{jk'}(a^{j'k}\hat{\ell}_{j'})_{k'}+a^{k'k}(a^{j'j}\hat{\ell}_{j'})_{k'}-(a^{jk}a^{j'k'}\hat{\ell}_{j'})_{k'}],
&\nonumber  \\
E^j=&\sum_{k=1}^na^{jk}(2\hat{\ell}_{k}\barr{\Phi}-2\hat{\ell}_{k}\hat{\ell}_{t}-\barr{\Phi}_k),
&\nonumber  \\
\text{and}\quad &F^{j}=\sum_{k=1}^n(a^{jk}\Phi_{k}-2a^{jk}\hat{\ell}_{tk}-a_{t}^{jk}\hat{\ell}_{k}-2a^{jk}\hat{\ell}_{k}\Phi).
&\nonumber
\end{flalign}
\end{lemma}
To state our Carleman estimates for (\ref{equationz}), we first introduce the weight function. To begin with, we introduce some auxiliary functions.
Let $G'$ be any given nonempty open subset of $G$ satisfying $G'\subset\subset G_{0}$. It can be seen from \cite{Fursikov1996Controllability}  that there exists a function $\beta\in C^4(\barr{G})$ such that
\begin{equation}\label{beta}
	\left\{
		\begin{aligned}
	&0<\beta(x)\leq 1\quad\forall x\in G,\\
	&\beta(x)=0\quad \forall x\in \partial G,\\
	&\inf_{G\backslash \barr{G'}}\{|\nabla\beta(x)|\}\geq \alpha_0 >0.
	    \end{aligned}
    \right.
\end{equation}
Without loss of generality, in what follows we assume that $0<T<1$. For some constant $m\geq 1$ and $\sigma \geq 2$, we define the following weight function depending on the time variable:
\begin{equation}\label{gamma}
	\left\{
		\begin{aligned}
	&\gamma(t)=1+\bigg(1-\frac{4t}{T}\bigg)^{\sigma}, &t\in &[0,T/4),\\
	&\gamma(t)=1,&t\in &[T/4,T/2),\\
	&\gamma(t)\ \text{is increasing on}\ [T/2,3T/4),\\
	&\gamma(t)=\frac{1}{(T-t)^m},& t\in &[3T/4,T],\\
	&\gamma(t)\in C^2([0,T)).
	    \end{aligned}
    \right.
\end{equation}
Set
\begin{equation}\label{99}
    \alpha(x):=e^{\mu(6m+\beta(x))}-\mu e^{\mu(6m+6)},\quad
	\varphi(t,x):=\gamma(t)\alpha(x),\quad
    \xi(t,x):=\gamma(t)e^{\mu(6m+\beta(x))},
\end{equation}
where $\mu$ is a positive parameter with $\mu\geq 1$ and $\sigma$ is chosen as
\begin{equation}\label{sigma}
\sigma=\lambda\mu^2e^{\mu (6m-4)}	
\end{equation}
for some parameter $\lambda\geq 1$.
We finally set the weight
\begin{equation}\label{theta}
	\theta :=e^{\ell},\ \text{where}\ \ell(t,x):=\lambda\varphi(t,x).
\end{equation}
It can be seen from the (\ref{theta})  that, compared with the classical weight function (see, e.g. \cite{Tang2009Null}), the main difference here is that the weight does not degenerate as $t\rightarrow 0^+$.

In the following, for any $k\in \mathbb{N}$, we denote by $\mathcal{O}(\mu^{k})$ a function of order $\mu^{k}$ for sufficiently large $\mu$, and by $\mathcal{O}_{\mu}(\lambda^{k})$ a function of order $\lambda^{k}$ for fixed $\mu$ and sufficiently large $\lambda$.

Now, we state the results of the Carleman estimate.
\begin{theorem}\label{carle1}
Assume that $\Xi\in L^2_{\mathbb{F}}(0,T;L^2(G;\mathbb{C}))$,  then there exist $\lambda_{0}>0$ and $\mu_{0}>0$ such that the unique solution $(z,Z)\in \mathcal{H}_{T}\times L^2_{\mathbb{F}}(0,T;L^2(G;\mathbb{C}))$ to $(\ref{equationz})$ with respect to $r_{T}\in L^2_{\mathcal{F}_{T}}(\Omega;L^{2}(G;\mathbb{C}))$ satisfies
    \begin{equation}\label{carest1}
    \begin{split}
	&\EE \int_{G}\lambda^2\mu^3e^{2\mu(6m+1)}\theta^2(0)|z(0)|^2\dx
	+\EE \int_{Q}\lambda\mu^2\xi\theta^2|\nabla z|^2\dxt\\
	&+\EE \int_{Q}\lambda^3\mu^4\xi^3\theta^2|z|^2\dxt \\
	\leq &
	C\bigg(\EE \int_{0}^{T}\int_{G_0} \lambda^3\mu^4\xi^3\theta^2|z|^2\dxt
	+\EE \int_{Q}\theta^2|\Xi|^2\dxt\\
	&+\EE \int_{Q}\lambda^2\mu^2\xi^3\theta^2|Z|^2\dxt
	\bigg),
	\end{split}
	\end{equation}	
for all $\lambda\geq\lambda_{0}$ and $\mu\geq\mu_{0}$, where  $C>0$ only depends on $G$ and $G_0$.
\end{theorem}

\begin{remark}
This type of Carleman estimate was first considered in \cite{Badra2016Local} to
deal with the local trajectory controllability for the incompressible Navier-Stokes equations.
Later, references \cite{Hernandez-Santamaria2023Global} and \cite{Zhang2024New} developed the ideas in \cite{Badra2016Local} to study the global null controllability of stochastic semilinear parabolic equations.
\end{remark}

\begin{proof}[\bf Proof of Theorem \ref{carle1}:]
The proof will be divided into several steps.

{\noindent\bf Step 1.}
We apply Lemma \ref{identity} with $\hat\theta=\theta$, $\hat\ell=\ell$ and set
\begin{equation}\label{Phi}
\Phi=2(a+\rmm{i} b)\sum_{j,k=1}^{n}(a^{jk}\ell_{j})_{k}.
\end{equation}
Integrating (\ref{identity1}) in $Q$, taking mathematical expectation on both sides, we conclude that
\begin{equation}\label{1}
\begin{split}
2\EE\int_{Q}\rmm{Re}({\theta}\barr{I_1}{\mathcal L}z)\dx
= &\ \EE\int_{Q}\dd M\dx
+\EE\int_{Q} B|v|^2\dxt
+\EE\int_{Q}\sum_{j,k=1}^nD^{jk}v_{j}\barr{v}_{k}\dxt
 \\
 & -2\EE\int_{Q}\sum_{j=1}^n[\rmm{Re}(aE^j\barr{v}v_j)-\rmm{Im}(bF^jv\barr{v}_j)]\dxt
 \\
\end{split}
\end{equation}
\begin{equation}\nonumber
\begin{split}
& +\EE\int_{Q}\sum_{k=1}^nV_{k}^k\dx
+\EE\int_{Q}\sum_{j,k=1}^naa^{jk}\dd v_j\dd \barr{v}_k\dx
\\
 &+\EE\int_{Q}({\ell}_t-aA)\dd v\dd \barr{v}\dx
-2b\EE\int_{Q}\sum_{j,k=1}^na^{jk}{\ell}_k\rmm{Im}(\dd v\dd \barr{v}_j)\dx \\
 & +2\EE\int_{Q}\bigg[b\sum_{j,k=1}^n(a^{jk}{\ell}_k)_j\rmm{Im}(v\dd \barr{v})+\rmm{Re}(\barr{\Phi}\barr{v}\dd v)\bigg]\dx
 \\
 & +2\EE\int_{Q}|I_1|^2\dxt
 =:\sum_{i=1}^{10}\mathcal{J}_{i}.
\end{split}
\end{equation}
In addition, it is easy to check that for any $j,k=1,\cdots,n$,
\begin{equation}\nonumber
\ell_{j}=\lambda\mu\xi\beta_{j},\quad
\ell_{jk}=\lambda\mu^2\xi\beta_{j}\beta_{k}+\lambda\mu\xi\beta_{jk},\quad
\ell_{t}=\lambda\gamma'(t)\alpha,
\end{equation}
which are useful in the remainder of the proof.

{\noindent\bf Step 2.} In this step, let us estimate every term on the right side of (\ref{1}).

{\it Estimate for $\mathcal{J}_1$.}
From the definitions of $v=\theta z$ and $\ell$, we easily see that $\lim_{t\rightarrow T-}\ell(t,\cdot)=-\infty$ and thus the term at $t=T$ vanishes. By expression of $M$, we have
\begin{equation}\label{2}
M|_0^T=-M(0)=
	-aA(0)|v(0)|^2
  -\sum_{j,k=1}^na^{jk}[-av_j\barr{v}_k+2b{\ell}_j\rmm{Im}(\barr{v}_kv)]\bigg|_{t=0}
  +{\ell}_{t}(0)|v(0)|^2.
\end{equation}

For the first term on the right-hand side of (\ref{2}), using the explicit expression of $A$ in (\ref{A}) and the fact $\gamma(0)=2$, we can obtain
\begin{equation}\label{9}
\begin{split}
A=&\sum_{j,k=1}^n(\lambda^2\mu^2\xi^2 a^{jk}\beta_j\beta_k
-\lambda\mu^2\xi a^{jk}\beta_j\beta_k
-\lambda\mu\xi a^{jk}_{k}\beta_{j}
-\lambda\mu\xi a^{jk} \beta_{jk})\\
=&\lambda^2\mu^2\xi^2\sum_{j,k=1}^na^{jk}\beta_j\beta_k+{\mathcal O}(\lambda)\mu^2\xi,
\end{split}
\end{equation}
which imply that
\begin{equation}\label{6}
\begin{split}
-aA(0)=&-a[4\lambda^2\mu^2e^{2\mu(6m+\beta)}\sum_{j,k=1}^na^{jk}(0)\beta_j\beta_k
 +{\mathcal O}(\lambda)\mu^2e^{\mu(6m+\beta)} ]\\
 \geq &-C\lambda^2\mu^2e^{2\mu(6m+\beta)}
 +{\mathcal O}(\lambda)\mu^2e^{\mu(6m+\beta)}.
\end{split}
\end{equation}

For the second one, by assumption $(H_{3})$ on the functions $a^{jk}$, we have
\begin{equation}\label{3}
	a\sum_{j,k=1}^na^{jk}v_{j}\barr{v}_k\bigg|_{t=0}\geq as_0|\nabla v(0)|^2.
\end{equation}
And, for any $\epsilon_1>0$, using Young inequality, we obtain
\begin{equation}\label{5}
\begin{split}
&\bigg|2b\sum_{j,k=1}^n a^{jk}\ell_{j}\rmm{Im}(\barr{v}_kv)\bigg|
=\bigg|2b\lambda\mu\xi\sum_{j,k=1}^n  a^{jk}\beta_{j}\rmm{Im}(\barr{v}_kv) \bigg|\\
&\leq \epsilon_1|\nabla v|^2+\frac{C}{\epsilon_1}\lambda^2\mu^2 \xi^2|v|^2.
\end{split}
\end{equation}
Choosing $\epsilon_1=\frac{as_0}{2}$ in (\ref{5}), we get
\begin{equation}\label{4}
\begin{split}
	\bigg|2b\sum_{j,k=1}^n a^{jk}\ell_{j}\rmm{Im}(\barr{v}_kv)\bigg|_{t=0}\bigg|
	\leq \frac{as_0}{2}|\nabla v(0)|^2
	+C \lambda^2\mu^2 e^{2\mu(6m+\beta)}|v(0)|^2,
\end{split}
\end{equation}
From (\ref{3}) and (\ref{4}), we get
\begin{equation}\label{7}
\begin{split}
    -\sum_{j,k=1}^na^{jk}[-av_j\barr{v}_k+2b{\ell}_j\rmm{Im}(\barr{v}_kv)]\bigg|_{t=0}
    \geq  \frac{as_0}{2}|\nabla v(0)|^2
    -C\lambda^2\mu^2 e^{2\mu(6m+\beta)}|v(0)|^2.
\end{split}
\end{equation}

For the third one, from the explicit expression of the function $\gamma(t)$ in (\ref{gamma}), we get
\begin{equation}\label{16}
\gamma'(t)=-\frac{4\sigma}{T}\bigg(1-\frac{4t}{T}\bigg)^{\sigma-1}	\quad \forall t\in [0,T/4].
\end{equation}
Using (\ref{sigma}) and the above expression, we have
\begin{equation}\label{8}
\begin{split}
\ell_{t}(0,x)=&-\frac{4}{T}\lambda^2\mu^2e^{\mu (6m-4)}\alpha(x)\\
=&\frac{4}{T}\lambda^2\mu^2e^{\mu (6m-4)}( \mu e^{\mu(6m+6)}-e^{\mu(6m+\beta(x))})\\
\geq &c_0\lambda^2\mu^3 e^{2\mu(6m+1)}
\end{split}
\end{equation}
for all $\mu> 1$ and some constant $c_0>0$ uniform with respect to $T$.

By (\ref{6}), (\ref{7}), (\ref{8}) and (\ref{2}), there exists $\mu_1>1$, such that for all $\mu\geq \mu_1>1$, it holds that
\begin{equation}\label{58}
M|_0^T \geq \frac{as_0}{2}|\nabla v(0)|^2
+c_1 \lambda^2\mu^3 e^{2\mu(6m+1)}|v(0)|^2
\end{equation}
for some constant $c_1>0$ only depending on $G$ and $G'$.
From (\ref{58}), we can easily get
\begin{equation}\label{39}
\mathcal{J}_{1}:=\EE\int_{G}M|_{0}^T\dx
\geq \frac{as_0}{2}\EE\int_{G}|\nabla v(0)|^2\dx
+c_1 \lambda^2\mu^3 e^{2\mu(6m+1)}\EE\int_{G}|v(0)|^2\dx.
\end{equation}

{\it Estimate for $\mathcal{J}_{2}$.} To abridge the notation, in what
follows, we set
\begin{equation}\nonumber
	\Psi=\sum_{j,k=1}^{n}a^{jk}\beta_j\beta_k.
\end{equation}

By the definition of $B$ in Lemma \ref{identity} and (\ref{Phi}), we have that
\begin{equation}\nonumber
\begin{split}
	B=&2(a^2+b^2)\sum_{j,k=1}^n(Aa^{jk}{\ell}_j)_k
   -aA_t
   -2aA\rmm{Re}\Phi
   -2bA\rmm{Im}\Phi
   \\
   &+2\rmm{Re}(\Phi{\ell}_t)
   -2|\Phi|^2
   +{\ell}_{tt}\\
   =&2(a^2+b^2)\sum_{j,k=1}^n A_{k}a^{jk}\ell_{j}
   +2(a^2+b^2)A\sum_{j,k=1}^n (a^{jk}\ell_{j})_{k}\\
\end{split}
\end{equation}
\begin{equation}\nonumber
\begin{split}
   &-4a^2A\sum_{j,k=1}^n(a^{jk}\ell_{j})_{k}
   -4b^2A\sum_{j,k=1}^n(a^{jk}\ell_{j})_{k}\\
   &-8(a^2+b^2)\bigg[\sum_{j,k=1}^n(a^{jk}\ell_{j})_{k}\bigg]^2
   -aA_t
   +4a\sum_{j,k=1}^n(a^{jk}\ell_{j})_{k}\ell_t
   +\ell_{tt}\\
   =&2(a^2+b^2)\bigg[\sum_{j,k=1}^n A_{k}a^{jk}\ell_{j}
   -A\sum_{j,k=1}^n(a^{jk}\ell_{j})_{k}\bigg]\\
   &-8(a^2+b^2)\bigg[\sum_{j,k=1}^n(a^{jk}\ell_{j})_{k}\bigg]^2
   -aA_t
   +4a\sum_{j,k=1}^n(a^{jk}\ell_{j})_{k}\ell_t
   +\ell_{tt}
   =:\sum_{j=1}^5I_{j}
\end{split}	
\end{equation}

Next, we estimate $I_j\ (j=1,\dots,5)$. Using (\ref{9}), we have
\begin{equation}\label{11}
A_k=2\lambda^2\mu^3\xi^2\beta_k\Psi
+{\mathcal O}(\lambda^2)\mu^2\xi^2
+{\mathcal O}(\lambda)\mu^3\xi.
\end{equation}
Further, using (\ref{11}) and (\ref{9}), we obtain respectively that
\begin{equation}\label{12}
\begin{split}
	\sum_{j,k=1}^nA_{k}a^{jk}\ell_{j}
	=2\lambda^3\mu^4\xi^3\Psi^2
	+{\mathcal O}(\lambda^3)\mu^3\xi^3
    +{\mathcal O}(\lambda^2)\mu^4\xi^2
\end{split}
\end{equation}
and
\begin{equation}\label{13}
    A\sum_{j,k=1}^n(a^{jk}\ell_j)_k	
    =\lambda^3\mu^4\xi^3\Psi^2
    +{\mathcal O}(\lambda^3)\mu^3\xi^3
    +{\mathcal O}(\lambda^2)\mu^4\xi^2.
\end{equation}
By (\ref{12}) and (\ref{13}), $I_1$ simplifies to
\begin{equation}\nonumber
\begin{split}
&I_1
   =2(a^2+b^2)\lambda^3\mu^4\xi^3\Psi^2
   +{\mathcal O}(\lambda^3)\mu^3\xi^3
    +{\mathcal O}(\lambda^2)\mu^4\xi^2,
\end{split}
\end{equation}
which imply that
\begin{equation}\label{22}
\begin{split}
\EE\int_{Q} I_1 |v|^2\dxt \geq &
2(a^2+b^2)\EE\int_{Q}  \lambda^3\mu^4\xi^3 |\Psi|^2 |v|^2 \dxt \\
&-C\EE\int_{Q} (\lambda^3\mu^3\xi^3+\lambda^2\mu^4\xi^2)|v|^2\dxt .
\end{split}
\end{equation}

For $I_2$, we easily have
\begin{equation}\nonumber
	I_2
    ={\mathcal O}(\lambda^2)\mu^4\xi^2,
\end{equation}
which imply that
\begin{equation}\label{23}
	\EE\int_{Q} I_2 |v|^2\dxt \geq
	-C\EE\int_{Q} \lambda^2\mu^4\xi^2|v|^2\dxt.
\end{equation}

Further, we estimate the third and the fourth terms together. Using (\ref{9}), we have
\begin{equation}\label{14}
\begin{split}
	A_t=&\sum_{j,k=1}^n(\lambda^2\mu^2\xi^2 a^{jk}\beta_j\beta_k
-\lambda\mu^2\xi a^{jk}\beta_j\beta_k
-\lambda\mu\xi a^{jk}_{k}\beta_{j}
-\lambda\mu\xi a^{jk} \beta_{jk})_t\\
=&2\lambda^2\mu^2\xi\xi_t\Psi
-\lambda\mu^2\xi_t\Psi
+{\mathcal O}(\lambda)\mu\xi_t
+{\mathcal O}(\lambda^2)\mu^2\xi^2\\
=&\frac{\gamma_t}{\gamma}\big[2\lambda^2\mu^2\xi^2\Psi
+{\mathcal O}(\lambda)\mu^2\xi\big]
+{\mathcal O}(\lambda^2)\mu^2\xi^2.
\end{split}
\end{equation}
In addition, we get
\begin{equation}\label{15}
\begin{split}
	\sum_{j,k=1}^n(a^{jk}\ell_{j})_{k}\ell_t
	=&\lambda^2\mu^2\xi\gamma_t\alpha\Psi
	+{\mathcal O}(\lambda^2)\mu\xi\gamma_t\alpha\\
	=&\frac{\gamma_t}{\gamma}\big[\lambda^2\mu^2\xi\varphi\Psi
	+{\mathcal O}(\lambda^2)\mu\xi\varphi\big].
\end{split}
\end{equation}
By (\ref{14}) and (\ref{15}), we obtain that
\begin{equation}\label{17}
\begin{split}
	I_3+I_4
   =W
+{\mathcal O}(\lambda^2)\mu^2\xi^2,
\end{split}	
\end{equation}
where
\begin{equation}\nonumber
\begin{split}
	W=-\frac{\gamma_t}{\gamma}\big[2a\lambda^2\mu^2\xi^2\Psi
+{\mathcal O}(\lambda)\mu^2\xi\big]
-\frac{\gamma_t}{\gamma}\big[4a\lambda^2\mu^2\xi(-\varphi)\Psi
	+{\mathcal O}(\lambda^2)\mu\xi(-\varphi)\big].
\end{split}	
\end{equation}

Further, we estimate the $W$. From the definition of $\gamma$, it is clear that $W$ vanishes on $(T/4,T/2)$. On $(T/2,T)$, we use the fact that there exists $C>0$ such that $|\gamma_t|\leq C|\gamma|^2$. Therefore, there exists a constant $C>0$ only depending on $G, G'$ such that
\begin{equation}\label{18}
|W|\leq C\lambda^2\mu^3\xi^3, \quad (t,x)\in (T/2,T)\times G ,	
\end{equation}
where we have used that $|\varphi\gamma|\leq \mu\xi^2$.
On $(0,T/4)$, noting the facts that $\mu\xi\leq|\varphi|$, $\gamma_t\leq 0$, $\varphi\leq 0$ and $\gamma\in[1,2]$, it holds that
\begin{equation}\label{19}
\begin{split}
	W\geq\ &4a\frac{|\gamma_t|}{\gamma}\lambda^2\mu^2\xi |\varphi|\Psi
	-C\frac{|\gamma_t|}{\gamma}\lambda^2\mu\xi|\varphi|,
	\quad (t,x)\in (0,T/4)\times G.
\end{split}
\end{equation}
By (\ref{17}), (\ref{18}) and (\ref{19}), we get
\begin{equation}\label{24}
\begin{split}
&\EE\int_{Q} (I_3+I_4) |v|^2\dxt
\geq \EE\int_{0}^{T/4}\int_{G} 4a\frac{|\gamma_t|}{\gamma}\lambda^2\mu^2\xi\Psi |\varphi||v|^2\dxt\\
&-C\EE\int_{0}^{T/4}\int_{G}\frac{|\gamma_t|}{\gamma}(\lambda\mu^2+\lambda^2\mu)\xi|\varphi||v|^2\dxt
-C\EE\int_{Q}\lambda^2\mu^3\xi^3 |v|^2\dxt .
\end{split}
\end{equation}

Further, we estimate $I_5$. Using the definition of $\gamma(t)$, it is not difficult to see that $|\gamma_{tt}|\leq C \lambda^2\mu^4e^{2\mu(6m-4)}$, thus
\begin{equation}\label{10}
|I_{5}|=|\ell_{tt}|=|\lambda\gamma_{tt}\alpha |\leq
C\lambda^3\mu^5 e^{2\mu(6m-4)} e^{\mu(6m+6)}
\leq C\lambda^3\mu^2 \xi^3,\quad t\in (0,T/4),
\end{equation}
where we have used that $\mu^3e^{-2\mu}<\frac{1}{2}$ for all $\mu>1$.
Noting that estimates $|\gamma_{tt}|\leq C\gamma^3$ and $|\varphi\gamma|\leq \mu\xi^2$, we obtain
\begin{equation}\label{21}
|I_{5}|=|\ell_{tt}|=\bigg|\lambda\frac{\gamma_{tt}}{\gamma}\varphi\bigg|
\leq C \lambda\mu\xi^3,\quad t\in (T/2,T).
\end{equation}
Since $\ell_{tt}$ vanishes for $t\in (T/4,T/2)$, from (\ref{10}) and (\ref{21}), we deduce that
\begin{equation}\label{25}
\EE\int_{Q} I_5|v|^2 \dxt \geq
-C\EE\int_{Q} \lambda^3\mu^2 \xi^3|v|^2\dxt.
\end{equation}

By (\ref{22}), (\ref{23}), (\ref{24}) and (\ref{25}), we obtain that
\begin{equation}\label{40}
\begin{split}
	\mathcal{J}_{2}:=&\EE\int_{Q}B|v|^2\dxt
	=\sum_{j=1}^5\EE\int_{Q}I_{j}|v|^2\dxt
\\
\geq &
	\EE\int_{0}^{T/4}\int_{G} 4a\frac{|\gamma_t|}{\gamma}\lambda^2\mu^2\xi\Psi |\varphi||v|^2\dxt
	\\
&-C\EE\int_{0}^{T/4}\int_{G}\frac{|\gamma_t|}{\gamma}(\lambda\mu^2+\lambda^2\mu)\xi|\varphi||v|^2\dxt
\\
&+2(a^2+b^2)\EE\int_{Q}\lambda^3\mu^4\xi^3 |\Psi|^2 |v|^2 \dxt
-C\EE\int_{Q} (\lambda^3\mu^3\xi^3+\lambda^2\mu^4\xi^2)|v|^2\dxt.
\end{split}
\end{equation}

{\it  Estimate for $\mathcal{J}_{3}$.}
By the definition of $D^{jk}$ in Lemma \ref{identity} and (\ref{Phi}), we find that
\begin{equation}\nonumber
\begin{split}
D^{jk}=&aa^{jk}_{t}
        +2b\rmm{Im}\Phi a^{jk}
        +2a\rmm{Re}\Phi a^{jk}\\
      &+2(a^2+b^2)\sum_{j',k'=1}^n[a^{jk'}(a^{j'k}{\ell}_{j'})_{k'}+a^{k'k}(a^{j'j}{\ell}_{j'})_{k'}-(a^{jk}a^{j'k'}{\ell}_{j'})_{k'}]\\
      =& 2(a^2+b^2)a^{jk}\lambda\mu^2\xi\Psi
      + 4(a^2+b^2)\lambda\mu^2\xi\sum_{j',k'=1}^{n}a^{jk'}a^{j'k}\beta_{j'}\beta_{k'}
      +{\mathcal O}(\lambda)\mu\xi.
\end{split}
\end{equation}
Hence, it holds that
\begin{equation}\label{41}
\begin{split}
	\mathcal{J}_{3}:=&\EE\int_{Q}\sum_{j,k=1}^nD^{jk}v_{j}\barr{v}_{k}\dxt
	\geq
	2(a^2+b^2)\EE\int_{Q}\lambda\mu^2\xi\Psi \sum_{j,k=1}^na^{jk}v_{j}\barr{v}_{k}\dxt\\
	&+4(a^2+b^2)\EE\int_{Q}\lambda\mu^2\xi\bigg|\sum_{j,k=1}^na^{jk}\beta_{j}v_{k}\bigg|^{2}\dxt
	-C\EE\int_{Q}\lambda\mu\xi|\nabla v|^2\dxt .
\end{split}	
\end{equation}

{\it  Estimate for $\mathcal{J}_{4}$.}
First, notice that for any $j=1,2,\dots,n$,
\begin{equation}\nonumber
	\rmm{Im}(bF^jv\barr{v}_j)=\rmm{Re}(\rmm{i}b\barr{F^j}\barr{v}{v}_j).
\end{equation}
By the definitions of $E^{j}$ and $F^{j}$ in Lemma \ref{identity} and (\ref{Phi}), we find that
\begin{equation}\label{29}
\begin{split}
	&-2\sum_{j=1}^n[\rmm{Re}(aE^j\barr{v}v_j)-\rmm{Im}(bF^jv\barr{v}_j)]
	=-2\sum_{j=1}^n\rmm{Re}[(aE^j-\rmm{i}b\barr{F^j})\barr{v}v_j]\\
	=&-4\sum_{j,k=1}^na^{jk}\ell_{k}\rmm{Re}[(a+\rmm{i}b)\barr{\Phi}\barr{v}v_{j}]
	+2\sum_{j,k=1}^n a^{jk}\rmm{Re}[(a+\rmm{i}b)\barr{\Phi}_k\barr{v}v_{j}]\\
	&+4\sum_{j,k=1}^n a a^{jk} \ell_{k}\ell_{t}\rmm{Re}(\barr{v}v_{j})
	-2\sum_{j,k=1}^n\rmm{Re}[\rmm{i}b(2a^{jk} \ell_{kt}+a_{t}^{jk}\ell_{k})\barr{v}v_{j}]\\
	=&-8(a^2+b^2)\sum_{j,k,j',k'=1}^n a^{jk}\ell_{k}(a^{j'k'}\ell_{j'})_{k'}\rmm{Re}(\barr{v}v_{j})\\
	&+4\sum_{j,k=1}^n aa^{jk}\ell_{k}\ell_{t}\rmm{Re}(\barr{v}v_{j})
	-4b\sum_{j,k=1}^na^{jk}\ell_{kt}\rmm{Re}(\rmm{i}\barr{v}v_{j})
	+\mathcal{O}(\lambda)\mu^3\xi |v||\nabla v|\\
\end{split}
\end{equation}
\begin{equation}\nonumber
\begin{split}
	=&-4(a^2+b^2)\sum_{j,k,j',k'=1}^n \bigg\{
	[a^{jk}\ell_{k}(a^{j'k'}\ell_{j'})_{k'}|v|^2]_{j}
	-[a^{jk}\ell_{k}(a^{j'k'}\ell_{j'})_{k'}]_{j}|v|^2\bigg\}\\
	&+2\sum_{j,k=1}^n(aa^{jk}\ell_{k}\ell_{t}|v|^2)_{j}
	-2\sum_{j,k=1}^n (aa^{jk}\ell_{k}\ell_{t})_{j}|v|^2
	-4b\sum_{j,k=1}^na^{jk}\ell_{kt}\rmm{Re}(\rmm{i}\barr{v}v_{j})\\
	&+\mathcal{O}(\lambda)\mu^3\xi |v||\nabla v|\\
	=&-2\sum_{j,k=1}^n\bigg[2(a^2+b^2)\sum_{j',k'=1}^na^{jk}\ell_{k}(a^{j'k'}\ell_{j'})_{k'}|v|^2
	-aa^{jk}\ell_{k}\ell_{t}|v|^2\bigg]_{j}\\
	&+\mathcal{O}(\lambda^2)\mu^4\xi^2 |v|^2
	+\mathcal{O}(\lambda)\mu^3\xi |v||\nabla v|\\
	&-2\sum_{j,k=1}^n (aa^{jk}\ell_{k}\ell_{t})_{j}|v|^2
	-4b\sum_{j,k=1}^na^{jk}\ell_{kt}\rmm{Re}(\rmm{i}\barr{v}v_{j}).
\end{split}	
\end{equation}

Further, we estimate the last two terms of the above expression. First, we have
\begin{equation}\nonumber
	-2\sum_{j,k=1}^n(aa^{jk}\ell_{k}\ell_{t})_{j}
	=-2a\frac{\gamma_{t}}{\gamma} \lambda^2\mu^2\xi\Psi\varphi
	-2a\frac{\gamma_{t}}{\gamma}\lambda^2\mu^2\xi^2\Psi.
\end{equation}
On $(0,T/4)$, since the second term of the above expression is positive, it holds that
\begin{equation}\label{20}
	-2\sum_{j,k=1}^n(aa^{jk}\ell_{k}\ell_{t})_{j}\geq
	-2a\frac{|\gamma_{t}|}{\gamma} \lambda^2\mu^2\xi\Psi|\varphi|,\quad (t,x)\in (0,T/4)\times G.
\end{equation}
On $(T/2,T)$, noting $|\gamma_{t}|\leq C|\gamma|^2$ and $|\gamma\varphi|\leq \mu \xi^2$, we have
\begin{equation}\label{26}
	\bigg|-2\sum_{j,k=1}^n(aa^{jk}\ell_{k}\ell_{t})_{j}\bigg|\leq
	C\lambda^2\mu^3\xi^3, \quad
	(t,x)\in (T/2,T)\times G.
\end{equation}
Since obviously it vanishes for $t\in (T/4,T/2)$, we can put estimates (\ref{20}) and (\ref{26}) together to deduce that
\begin{equation}\label{30}
\begin{split}
	\intt -2\sum_{j,k=1}^n(aa^{jk}\ell_{k}\ell_{t})_{j}|v|^2\dt
	\geq &\int_0^{T/4} -2a\frac{|\gamma_{t}|}{\gamma} \lambda^2\mu^2\xi\Psi|\varphi||v|^2\dt \\
	&-C \intt \lambda^2\mu^3\xi^3 |v|^2\dt .
\end{split}
\end{equation}

Next, we have
\begin{equation}\nonumber
	-4b\sum_{j,k=1}^na^{jk}\ell_{kt}\rmm{Re}(\rmm{i}\barr{v}v_{j})
	=-4b\frac{|\gamma_{t}|}{\gamma} \lambda\mu\xi\rmm{Re}\bigg[\rmm{i}\bigg(\sum_{j,k=1}^na^{jk}\beta_{k}v_{j}\bigg)\barr{v}\bigg].
\end{equation}
On $(T/2,T)$, due to $|\gamma_t|\leq C\gamma^2$, we obtain
\begin{equation}\label{27}
	\bigg|-4b\sum_{j,k=1}^na^{jk}\ell_{kt}\rmm{Re}(\rmm{i}\barr{v}v_{j})\bigg|
	\leq C\lambda\mu\xi^2\bigg|\sum_{j,k=1}^na^{jk}\beta_{k}v_{j}\bigg||v|.
\end{equation}
On $(0,T/4)$, by (\ref{16}), it holds that
\begin{equation}\label{28}
\begin{split}
	\bigg|-4b\sum_{j,k=1}^na^{jk}\ell_{kt}\rmm{Re}(\rmm{i}\barr{v}v_{j})\bigg|
	\leq &C\lambda^2\mu^3\xi e^{\mu(6m-4)}\bigg|\sum_{j,k=1}^na^{jk}\beta_{k}v_{j}\bigg||v|\\
	\leq &C\lambda^2\xi^2\bigg|\sum_{j,k=1}^na^{jk}\beta_{k}v_{j}\bigg||v|,
\end{split}
\end{equation}
where we also have used $\mu^3e^{-2\mu}\leq \frac{1}{2}$ for all $\mu>1$.
Similarly, it also vanishes for $t\in (T/4,T/2)$, using (\ref{27}) and (\ref{28}), we have
\begin{equation}\label{31}
\intt 	-4b\sum_{j,k=1}^na^{jk}\ell_{kt}\rmm{Re}(\rmm{i}\barr{v}v_{j})\dt
\geq -C\intt \lambda^2\mu\xi^2\bigg|\sum_{j,k=1}^na^{jk}\beta_{k}v_{j}\bigg||v|\dt .
\end{equation}

By (\ref{29}), (\ref{30}), (\ref{31}) and $v=0$ on $\Sigma$, we have
\begin{equation}\label{42}
\begin{split}
\mathcal{J}_{4}:=&\EE\int_{Q}-2\sum_{j=1}^n[\rmm{Re}(aE^j\barr{v}v_j)-\rmm{Im}(bF^jv\barr{v}_j)]	\dxt \\
\geq
	&-\EE\int_0^{T/4}\int_{G} 2a\frac{|\gamma_{t}|}{\gamma} \lambda^2\mu^2\xi\Psi|\varphi||v|^2\dxt \\
	&-C\EE\int_{Q}(\lambda^2\mu^4\xi^2
	+\lambda^2\mu^3\xi^3
	+\lambda^3\mu^2\xi^3) |v|^2\dxt \\
	&-C\EE\int_{Q}\mu^2 |\nabla v|^2 \dxt
	-C\EE\int_{Q} \lambda\xi\bigg|\sum_{j,k=1}^na^{jk}\beta_{k}v_{j}\bigg|^2\dxt.
\end{split}
\end{equation}

{\it  Estimate for $\mathcal{J}_{5}$.}
Noting that $v=0$ on $\Sigma$ and $\beta=0$ on $\Sigma$, for any $j=1,\dots,n$, we obtain
\begin{equation}\nonumber
	v_j=\frac{\partial v}{\partial \nu}\nu_{j},\quad
	\beta_j=\frac{\partial \beta}{\partial \nu}\nu_{j},
\end{equation}
where $\nu=(\nu_1,\dots,\nu_n)$ denotes the unit outer normal vector on $\Gamma$.
By the definition of $V^k$ in Lemma \ref{identity}, for any $k=1,\dots,n$, it holds that

\begin{equation}\nonumber
\begin{split}
V^k|_{\Sigma}=&	2(a^2+b^2)\sum_{j,j',k'=1}^n[a^{jk}a^{j'k'}{\ell}_{j}v_{j'}\barr{v}_{k'}-a^{jk'}a^{j'k}{\ell}_{j}(v_{j'}\barr{v}_{k'}+\barr{v}_{j'}v_{k'})]\dt\bigg|_{\Sigma}
\\
	=&-2(a^2+b^2)\sum_{j,j',k'=1}^n\lambda\mu\xi\frac{\partial \beta}{\partial \nu}\bigg|\frac{\partial v}{\partial \nu}\bigg|^2a^{jk}a^{j'k'}\nu_{j}\nu_{j'}\nu_{k'}\dt \bigg|_{\Sigma}.
\end{split}
\end{equation}
Therefore, we find that
\begin{equation}\label{43}
\begin{split}
	\mathcal{J}_{5}:=&\EE\int_{Q}\sum_{k=1}^nV_{k}^k\dx
	=\EE\int_{\Sigma} \sum_{k=1}^n V^{k} \nu_{k}\dd \sigma\\
	=&-2(a^2+b^2)\EE\int_{\Sigma}\lambda\mu\xi\frac{\partial \beta}{\partial \nu}\bigg|\frac{\partial v}{\partial \nu}\bigg|^2\sum_{j,j',k',k=1}^na^{jk}a^{j'k'}\nu_{j}\nu_{j'}\nu_{k'}\nu_{k}\dd \sigma
	\geq 0.
\end{split}
\end{equation}

{\it Estimate for $\mathcal{J}_{6},\mathcal{J}_{7}$ and $\mathcal{J}_{8}$.}
By the first equation of (\ref{equationz}), condition $(H_3)$ and the fact $v=\theta z$, we have
\begin{equation}\nonumber
\begin{split}
	\mathcal{J}_{6}:=&
	\EE\int_{Q}\sum_{j,k=1}^naa^{jk}\dd v_j\dd \barr{v}_k\dx
	\\
 =&\EE\int_{Q}\sum_{j,k=1}^naa^{jk}\theta^2
 \dd(\ell_{j}z+z_{j}) \dd(\ell_{k}\barr{z}+\barr{z}_{k}) \dx \\
 \geq &as_0\EE\int_{Q}\theta^2|\nabla Z+\nabla\ell Z|^2 \dxt \\
 \geq &\frac{1}{2}as_0\EE\int_{Q}\theta^2|\nabla Z|^2\dxt
 -C\EE\int_{Q}\lambda^2\mu^2\xi^2\theta^2|Z|^2\dxt .
\end{split}
\end{equation}
Similar to the previous analysis of $\ell_{t}$ and (\ref{9}), we have
\begin{equation}\nonumber
\begin{split}
	|\mathcal{J}_{7}|=
	\bigg|\EE\int_{Q}({\ell}_t-aA)\dd v\dd \barr{v}\dx \bigg|
	\leq & \EE\int_{Q}\bigg|\frac{\gamma_{t}}{\gamma} \lambda\varphi-aA \bigg|\theta^2|Z|^2\dxt\\
	\leq & C \EE\int_{Q}\lambda^2\mu^2\xi^3\theta^2|Z|^2\dxt  .
\end{split}
\end{equation}
Further, we have
\begin{equation}\nonumber
\begin{split}
\mathcal{J}_{8}:=&
-2b\EE\int_{Q}	\sum_{j,k=1}^n a^{jk}\ell_{k}\rmm{Im}(\dd v\dd\barr{v}_{j})\dx
\\
=&-2b\EE\int_{Q} \sum_{j,k=1}^n a^{jk}\beta_{k} \lambda\mu\xi \theta^2 \rmm{Im}[\dd z\dd (\ell_{j}\barr{z}+\barr{z}_{j})]\dx\\
\geq &-C\EE\int_{Q}\theta^2(\lambda^2\mu^2\xi^2|Z|^2+\lambda\mu\xi|Z||\nabla Z|)\dxt \\
\geq & -\frac{1}{2}as_0\EE\int_{Q}\theta^2|\nabla Z|^2\dxt
-C\EE\int_{Q}\theta^2\lambda^2\mu^2\xi^2|Z|^2\dxt .
\end{split}
\end{equation}

Therefore, it holds that
\begin{equation}\label{44}
\begin{split}
	\mathcal{J}_{6}
	+\mathcal{J}_{7}
	+\mathcal{J}_{8}
 \geq -C \EE\int_{Q}\lambda^2\mu^2\xi^3\theta^2|Z|^2\dxt.
\end{split}	
\end{equation}

{\it Estimate for $\mathcal{J}_{9}$.}
First, noting that the fact $\rmm{Im}(c)=\rmm{Re}(\rmm{i}\barr{c})$ for any $c\in \mathbb{C}$, we obtain that
\begin{equation}\label{33}
\begin{split}
&2b\sum_{j,k=1}^n(a^{jk}{\ell}_k)_j\rmm{Im}(v\dd \barr{v})+2\rmm{Re}(\barr{\Phi}\barr{v}\dd v)\\
=&2\sum_{j,k=1}^n(a^{jk}{\ell}_k)_j\rmm{Re}[(2a-\rmm{i}b)\barr{v}\dd v]\\
=&2a\sum_{j,k=1}^n(a^{jk}{\ell}_k)_j(v\dd \barr{v}+\barr{v}\dd v)
+\rmm{i}b\sum_{j,k=1}^n(a^{jk}{\ell}_k)_j(v\dd \barr{v}-\barr{v}\dd v)\\
\end{split}
\end{equation}
\begin{equation}\nonumber
\begin{split}
=&2a\dd \bigg[\sum_{j,k=1}^n(a^{jk}{\ell}_k)_j |v|^2\bigg]
-2a\sum_{j,k=1}^n(a^{jk}{\ell}_k)_{jt}|v|^2\dt
-2a\sum_{j,k=1}^n(a^{jk}{\ell}_k)_j\dd v\dd \barr{v} \\
&+\rmm{i}b\sum_{j,k=1}^n(a^{jk}{\ell}_k)_j(v\dd \barr{v}-\barr{v}\dd v).
\end{split}
\end{equation}

It is easy to check that
\begin{equation}\label{34}
\begin{split}
	2a\sum_{j,k=1}^n(a^{jk}{\ell}_k)_j |v|^2\bigg|_0^T
	=&\bigg[-4a\lambda\mu^2e^{\mu (6m+\beta)}\sum_{j,k=1}^na^{jk}\beta_{j}\beta_{k}+\mathcal{O}(\lambda)\mu e^{\mu (6m+\beta)}\bigg] |v(0)|^2\\
	\geq &-C\lambda\mu^2 e^{\mu (6m+\beta)}|v(0)|^2.
\end{split}
\end{equation}

Further, we have
\begin{equation}\label{32}
\begin{split}
	-2a\sum_{j,k=1}^n(a^{jk}{\ell}_k)_{jt}
	=-2a\frac{\gamma_{t}}{\gamma}[\Psi\lambda\mu^2\xi+\mathcal{O}(\lambda)\mu\xi]+\mathcal{O}(\lambda)\mu^2\xi.
\end{split}
\end{equation}
Combining with (\ref{32}), similar to the analysis for $W$ in Estimate for $\mathcal{J}_{2}$, we obtain
\begin{equation}\label{35}
\begin{split}
&-2a\EE\int_{Q} \sum_{j,k=1}^n(a^{jk}{\ell}_k)_{jt}|v|^2\dxt\\
\geq
&-C\EE\int_{0}^{T/4}\int_{G}\frac{|\gamma_t|}{\gamma}\lambda\mu^2\xi |\varphi||v|^2\dxt
-C\EE\int_{Q}\lambda\mu^2\xi^2 |v|^2\dxt.
\end{split}
\end{equation}

Further, we have
\begin{equation}\label{36}
\begin{split}
	-2a\sum_{j,k=1}^n(a^{jk}{\ell}_k)_j\dd v\dd \barr{v}
	=\mathcal{O}(\lambda)\mu^2\xi \theta^2|Z|^2\dt .
\end{split}	
\end{equation}

Integtating (\ref{33}) on $Q$, taking mathematical expectation on both sides and noting (\ref{34}), (\ref{35}) and (\ref{36}), we conclude that
\begin{equation}\label{37}
\begin{split}
\mathcal{J}_{9}:=
&\EE \int_{Q}2b\sum_{j,k=1}^n(a^{jk}{\ell}_k)_j\rmm{Im}(v\dd \barr{v})+2\rmm{Re}(\barr{\Phi}\barr{v}\dd v)\dx\\
\geq &-C\EE \int_{G} \lambda\mu^2 e^{\mu (6m+\beta)}|v(0)|^2\dx
-C\EE\int_{0}^{T/4}\int_{G}\frac{|\gamma_t|}{\gamma}\lambda\mu^2\xi |\varphi||v|^2\dxt \\
&-C\EE\int_{Q}\lambda\mu^2\xi^2 |v|^2\dxt
-C\EE\int_{Q}\lambda\mu^2\xi \theta^2|Z|^2\dxt \\
&+\EE \int_{Q}\rmm{i}b\sum_{j,k=1}^n(a^{jk}{\ell}_k)_j(v\dd \barr{v}-\barr{v}\dd v)\dx .
\end{split}
\end{equation}

Next, we estimate the last term of (\ref{37}). By the same analysis as (2.23) and (2.24) in \cite{Fu2017Controllability}, replacing $\psi$, $\varphi$, $f$, $z$, $y$ and $Y$ in \cite{Fu2017Controllability} by $\beta$, $\xi$, $\Xi $, $v$, $z$ and $Z$ respectively, we obtain that
\begin{equation}\nonumber
\begin{split}
	&\EE \int_{Q}\rmm{i}b\sum_{j,k=1}^n(a^{jk}{\ell}_k)_j(v\dd \barr{v}-\barr{v}\dd v)\dx\\
	\geq &-C\EE \int_{Q}\theta^2|\Xi|^2\dxt
	-C|ab|\EE \int_{Q}\lambda\mu^3\xi|v||\nabla v|\dxt \\
	&+(-a^2+b^2)\EE \int_{Q}\lambda^3\mu^4\xi^3|\Psi|^2|v|^2\dxt
	-Cb^2 \EE \int_{Q} (\lambda^2\mu^4\xi^2+\lambda^3\mu^3\xi^3)|v|^2\dxt \\
	&-2b^2\EE \int_{Q} \lambda\mu^2\xi\Psi \sum_{j,k=1}^n a^{jk}v_{k} \barr{v}_{j}\dxt
	-(a^2+b^2)\EE \int_{Q} \lambda\mu^2\xi \bigg|\sum_{j,k=1}^n a^{jk}\beta_{j}v_{k} \bigg|^2 \dxt \\
	&-C(a^2+b^2)\EE \int_{Q} [\lambda\mu\xi|\nabla v|^2+\lambda^3\mu^3\xi^3|v|^2]\dxt,
\end{split}	
\end{equation}
which imply that
\begin{equation}\label{38}
\begin{split}
	&\EE \int_{Q}\rmm{i}b\sum_{j,k=1}^n(a^{jk}{\ell}_k)_j(v\dd \barr{v}-\barr{v}\dd v)\dx\\
	\geq &-C\EE \int_{Q}\theta^2|\Xi|^2\dxt
	+(-a^2+b^2)\EE \int_{Q}\lambda^3\mu^4\xi^3|\Psi|^2|v|^2\dxt  \\
	&-2b^2\EE \int_{Q} \lambda\mu^2\xi\Psi \sum_{j,k=1}^n a^{jk}v_{k} \barr{v}_{j}\dxt
	-(a^2+b^2)\EE \int_{Q} \lambda\mu^2\xi \bigg|\sum_{j,k=1}^n a^{jk}\beta_{j}v_{k} \bigg|^2 \dxt \\
	&-C \EE \int_{Q} (\lambda^2\mu^4\xi^2+\lambda^3\mu^3\xi^3)|v|^2\dxt
	-C\EE \int_{Q} (\lambda\mu\xi+\mu^2)|\nabla v|^2\dxt.
\end{split}	
\end{equation}
By (\ref{37}) and (\ref{38}), we have
\begin{equation}\label{45}
\begin{split}
	\mathcal{J}_{9}
\geq &-C\EE \int_{G} \lambda\mu^2 e^{\mu (6m+\beta)}|v(0)|^2\dx
-C\EE\int_{0}^{T/4}\int_{G}\frac{|\gamma_t|}{\gamma}\lambda\mu^2\xi |\varphi||v|^2\dxt \\
&-C\EE\int_{Q}\lambda\mu^2\xi \theta^2|Z|^2\dxt
-C\EE \int_{Q}\theta^2|\Xi|^2\dxt \\
	&+(-a^2+b^2)\EE \int_{Q}\lambda^3\mu^4\xi^3|\Psi|^2|v|^2\dxt  \\
	&-2b^2\EE \int_{Q} \lambda\mu^2\xi\Psi \sum_{j,k=1}^n a^{jk}v_{k} \barr{v}_{j}\dxt
	-(a^2+b^2)\EE \int_{Q} \lambda\mu^2\xi \bigg|\sum_{j,k=1}^n a^{jk}\beta_{j}v_{k} \bigg|^2 \dxt \\
	&-C \EE \int_{Q} (\lambda^2\mu^4\xi^2+\lambda^3\mu^3\xi^3)|v|^2\dxt
	-C\EE \int_{Q} (\lambda\mu\xi+\mu^2)|\nabla v|^2\dxt.
\end{split}	
\end{equation}

{\noindent\bf Step 3: Towards the Carleman estimate.}
Combining (\ref{39}), (\ref{40}), (\ref{41}), (\ref{42}), (\ref{43}), (\ref{44}), (\ref{45}) with (\ref{1}),
we obtain that
\begin{equation}\label{46}
\begin{split}
&2\EE\int_{Q}\rmm{Re}({\theta}\barr{I_1}{\mathcal L}z)\dx
\geq 2\EE\int_{Q}|I_1|^2\dxt
 +c_1\EE\int_{G}  \lambda^2\mu^3 e^{2\mu(6m+1)}|v(0)|^2 \dx \\
&\quad
+2a\EE\int_{0}^{T/4}\int_{G} \frac{|\gamma_t|}{\gamma}\lambda^2\mu^2\xi\Psi |\varphi||v|^2\dxt
 +(a^2+3b^2)\EE\int_{Q}  \lambda^3\mu^4\xi^3 |\Psi|^2 |v|^2 \dxt
 \\
 &\quad +2a^2\EE\int_{Q}\lambda\mu^2\xi\Psi \sum_{j,k=1}^na^{jk}v_{j}\barr{v}_{k}\dxt
 +3(a^2+b^2)\EE\int_{Q}\lambda\mu^2\xi\bigg|\sum_{j,k=1}^na^{jk}\beta_{j}v_{k}\bigg|^{2}\dxt
 \\
 &\quad -C\bigg[\EE \int_{G} \lambda\mu^2 e^{\mu (6m+1)}|v(0)|^2\dx
+\EE\int_{0}^{T/4}\int_{G}\frac{|\gamma_t|}{\gamma}(\lambda\mu^2+\lambda^2\mu)\xi|\varphi||v|^2\dxt
\\
&\quad +\EE\int_{Q} (\lambda^3\mu^3\xi^3+\lambda^2\mu^4\xi^2)|v|^2\dxt
+\EE \int_{Q} (\lambda\mu\xi+\mu^2)|\nabla v|^2\dxt
\\
&\quad +\EE\int_{Q} \lambda\xi\bigg|\sum_{j,k=1}^na^{jk}\beta_{j}v_{k}\bigg|^2\dxt
+\EE \int_{Q}\lambda^2\mu^2\xi^3\theta^2|Z|^2\dxt
+\EE \int_{Q}\theta^2|\Xi|^2\dxt\bigg] .
\end{split}
\end{equation}
By (\ref{equationz}), we easily have
\begin{equation}\label{47}
	2\EE\int_{Q}\rmm{Re}({\theta}\barr{I_1}{\mathcal L}z)\dx
	\leq \EE\int_{Q}|I_1|^2\dxt
	+\EE\int_{Q}\theta^2|\Xi|^2\dxt .
\end{equation}
By (\ref{beta}), ($H_3$) and (\ref{47}), we can easily show that (\ref{46}) imply that
\begin{equation}\nonumber
\begin{split}
&\EE\int_{Q}|I_1|^2\dxt
 +c_1\EE\int_{G}  \lambda^2\mu^3 e^{2\mu(6m+1)}|v(0)|^2 \dx
 \\
 &+2as_0\alpha_0^2\EE\int_{0}^{T/4}\int_{G} \frac{|\gamma_t|}{\gamma}\lambda^2\mu^2\xi  |\varphi||v|^2\dxt
 +(a^2+3b^2)s_0^2\alpha_0^4\EE\int_{Q}  \lambda^3\mu^4\xi^3  |v|^2 \dxt
 \\
&+2a^2s_0^2\alpha_0^2\EE\int_{Q}\lambda\mu^2\xi  |\nabla v|^2\dxt
+3(a^2+b^2)\EE\int_{Q}\lambda\mu^2\xi\bigg|\sum_{j,k=1}^na^{jk}\beta_{j}v_{k}\bigg|^{2}\dxt\\
 \leq\
 &C\bigg[\EE \int_{G} \lambda\mu^2 e^{\mu (6m+1)}|v(0)|^2\dx
 +\EE\int_{0}^{T/4}\int_{G}\frac{|\gamma_t|}{\gamma}(\lambda\mu^2+\lambda^2\mu)\xi|\varphi||v|^2\dxt  \\
&+\EE\int_{Q} (\lambda^3\mu^3\xi^3+\lambda^2\mu^4\xi^2)|v|^2\dxt
+\EE \int_{Q} (\lambda\mu\xi+\mu^2)|\nabla v|^2\dxt\\
&+\EE\int_{Q} \lambda\xi\bigg|\sum_{j,k=1}^na^{jk}\beta_{j}v_{k}\bigg|^2\dxt
+\EE \int_{Q}\lambda^2\mu^2\xi^3\theta^2|Z|^2\dxt
+\EE\int_{Q}\theta^2|\Xi|^2\dxt
\\
&+\EE\int_{0}^{T/4}\int_{G'} \frac{|\gamma_t|}{\gamma}\lambda^2\mu^2\xi  |\varphi||v|^2\dxt
+\EE\intt\int_{G'}  \lambda^3\mu^4\xi^3  |v|^2 \dxt
\\
&+\EE\intt\int_{G'}\lambda\mu^2\xi |\nabla v|^2 \dxt
\bigg].
\end{split}
\end{equation}
Hence, there exists a constant $\mu_1>1$ such that for all $\mu>\mu_1$, there exists a constant $\lambda_1=\lambda_1(\mu)$ such that for any $\lambda>\lambda_1$, it holds that
\begin{equation}\nonumber
\begin{split}
&\EE\int_{G}  \lambda^2\mu^3 e^{2\mu(6m+1)}|v(0)|^2 \dx
+\EE\int_{0}^{T/4}\int_{G} \lambda^2\mu^2\xi  |\varphi||\gamma_t||v|^2\dxt\\
&+\EE\int_{Q}  \lambda^3\mu^4\xi^3  |v|^2 \dxt
+\EE\int_{Q}\lambda\mu^2\xi  |\nabla v|^2\dxt\\
\leq\
 &C\bigg[
\EE \int_{Q}\lambda^2\mu^2\xi^3\theta^2|Z|^2\dxt
+\EE\int_{Q}\theta^2|\Xi|^2\dxt
+\EE\int_{0}^{T/4}\int_{G'} \lambda^2\mu^2\xi  |\varphi| |\gamma_t| |v|^2\dxt
\\
&+\EE\intt\int_{G'}  \lambda^3\mu^4\xi^3  |v|^2 \dxt
+\EE\intt\int_{G'}\lambda\mu^2\xi |\nabla v|^2 \dxt
\bigg].
\end{split}
\end{equation}
By the definition of $v$, this implies that
\begin{equation}\label{48}
\begin{split}
&\EE\int_{G} \theta^2(0) \lambda^2\mu^3 e^{2\mu(6m+1)}|z(0)|^2 \dx
+\EE\int_{0}^{T/4}\int_{G}\theta^2 \lambda^2\mu^2\xi  |\varphi||\gamma_t||z|^2\dxt\\
&+\EE\int_{Q} \theta^2 \lambda^3\mu^4\xi^3  |z|^2 \dxt
+\EE\int_{Q}\theta^2\lambda\mu^2\xi  |\nabla z|^2\dxt
\\
\leq\
&C\bigg[
\EE \int_{Q}\lambda^2\mu^2\xi^3\theta^2|Z|^2\dxt
+\EE\int_{Q}\theta^2|\Xi|^2\dxt
+\EE\int_{0}^{T/4}\int_{G'}\theta^2 \lambda^2\mu^2\xi  |\varphi| |\gamma_t| |z|^2\dxt
\\
&+\EE\intt\int_{G'} \theta^2 \lambda^3\mu^4\xi^3  |z|^2 \dxt
+\EE\intt\int_{G'}\theta^2\lambda\mu^2\xi |\nabla z|^2 \dxt
\bigg].
\end{split}
\end{equation}
{\noindent\bf Step 4: Last arrangements and conclusion.}
We choose a real-valued function $\zeta\in C_0^{\infty}(G_0)$ such that $\zeta=1$ in $G'$ and $\frac{|\nabla \zeta|}{\zeta^{1/2}}\in L^{\infty}(G_0)$.
By the It\^o formula, we calculate
\begin{equation}\nonumber
\dd(\xi\theta^2\zeta |z|^2)
=\xi\theta^2\zeta z\dd \barr{z}
+\xi\theta^2\zeta \barr{z}\dd z
+(\xi\theta^2)_{t}\zeta|z|^2\dt
+\xi\theta^2\zeta|\dd z|^2.
\end{equation}
and thus, using the equation (\ref{equationz}) and $\zeta\in C^{\infty}_{0}(G_{0})$, integrating it over $Q$ and taking the mathematical expectation, we have
\begin{equation}\label{59}
\begin{split}
&2a\EE\int_{Q}\xi\theta^2\zeta\sum_{j,k=1}^na^{jk}z_{j}\barr{z}_{k}\dxt
+\EE\int_{Q}(\xi\theta^2)_{t}\zeta|z|^2\dxt
\\
&+\EE\int_{G}\xi(0,x)\theta^2(0,x)\zeta(x)|z(0,x)|^2 \dx
+\EE\int_{Q}\xi\theta^2\zeta|Z|^2\dxt
\\
=& -2\EE\int_{Q}\rmm{Re}[(a-\rmm{i}b)\barr{z}\sum_{j,k=1}^{n}a^{jk}z_{j}(\xi\theta^2\zeta)_{k}]\dxt
-2\EE\int_{Q}\xi\theta^2\zeta \rmm{Re}(z\barr{\Xi})\dxt.
\end{split}	
\end{equation}
We readily see that the last two terms on the left-hand side of (\ref{59}) are positive, so they can be dropped. Also, notice that using the properties of $\zeta$ and $(H_{3})$, the first term gives the local terms containing $|\nabla z|^2$. For the second term on the left-hand side of (\ref{59}), similar to step 2 above, we can analyze it on different time intervals. In summary, multiplying  both sides of (\ref{59}) by $\lambda\mu^2$, by the properties of $\zeta$, Young inequality and simple calculation, we have
\begin{equation}\label{49}
\begin{split}
&\EE\int_{0}^{T/4}\int_{G'}\theta^2 \lambda^2\mu^2\xi  |\varphi| |\gamma_t| |z|^2\dxt
+\EE\intt\int_{G'}\theta^2\lambda\mu^2\xi |\nabla z|^2 \dxt\\
\leq
&C\EE\intt\int_{G_0} \theta^2 \lambda^3\mu^4\xi^3  |z|^2 \dxt
+C\EE\int_{Q}\theta^2|\Xi|^2\dxt
\\
&+C\EE\int_{Q} \theta^2 \lambda^2\mu^4\xi^2  |z|^2 \dxt
+C\EE\int_{Q} \theta^2 \mu^2  |\nabla z|^2 \dxt.
\end{split}
\end{equation}
Combining (\ref{48}) with (\ref{49}), we obtain
\begin{equation}\label{60}
\begin{split}
&\EE\int_{G} \theta^2(0) \lambda^2\mu^3 e^{2\mu(6m+1)}|z(0)|^2 \dx
+\EE\int_{0}^{T/4}\int_{G}\theta^2 \lambda^2\mu^2\xi  |\varphi||\gamma_t||z|^2\dxt
\\
&+\EE\int_{Q} \theta^2 \lambda^3\mu^4\xi^3  |z|^2 \dxt
+\EE\int_{Q}\theta^2\lambda\mu^2\xi  |\nabla z|^2\dxt\\
\leq
&C\bigg[
\EE \int_{Q}\lambda^2\mu^2\xi^3\theta^2|Z|^2\dxt
+\EE\int_{Q}\theta^2|\Xi|^2\dxt
+\EE\intt\int_{G_{0}} \theta^2 \lambda^3\mu^4\xi^3  |z|^2 \dxt
\bigg]
\\
&+C\EE\int_{Q} \theta^2 \lambda^2\mu^4\xi^2  |z|^2 \dxt
+C\EE\int_{Q} \theta^2 \mu^2  |\nabla z|^2 \dxt.
\end{split}
\end{equation}
Finally, from (\ref{60}), we can see that there exist two constants $\mu_{0}>\mu_{2}$, $\lambda_{0}>\lambda_{1}$, such that (\ref{carest1}) holds for all $\mu\geq\mu_{0}$, $\lambda\geq\lambda_{0}$.
Thus, the proof of Theorem \ref{carle1} is complete.
\end{proof}

\subsection{The proof of Theorem \ref{control}}
In this section, we will give the proof of Theorem \ref{control}. Applying Theorem \ref{carle1}, we first  establish a controllability result for a linear forward stochastic complex Ginzburg-Landau equation with one source term and two controls, that is,
\begin{equation}\label{equationy3}
	\left\{
		\begin{aligned}
	&\dd y-(a+\rmm{i} b)\sum_{j,k=1}^{n}(a^{jk}y_j)_k\dt=(F+\chi_{G_0}h)\dt+H\dd B(t) &\textup{in}\ &Q,\\
    &y=0 &\textup{on}\ &\Sigma,\\
    &y(0)=y_0 &\textup{in}\ &G,
            \end{aligned}
    \right.
\end{equation}
where $y=y(t,x)$ denotes the state variable associated to the initial state $y_{0}\in L^2_{\mathcal{F}_{0}}(\Omega;L^2(G;\\\mathbb{C}))$, the control pair $(h,H)\in L^2_{\mathbb{F}}(0,T;L^2(G_{0};\mathbb{C}))\times L^2_{\mathbb{F}}(0,T;L^2(G;\mathbb{C}))$.
Observe that given the aforementioned regularity on the controls and source term, one can easily show that system $(\ref{equationy3})$ admits a unique solution $y\in \mathcal{H}_{T}$ (see e.g., \cite[Theorem 3.24]{Lu2021Mathematical}).

We define the space
\begin{equation}\nonumber
\begin{split}
	\mathscr{S}_{\lambda,\mu}=\bigg\{ F\in L^2_\mathbb{F}(0,T;L^2(G;\mathbb{C})) \bigg|
	\bigg(\EE \int_{Q}\theta^{-2}\lambda^{-3}\mu^{-4}\xi^{-3}|F|^2\dxt\bigg)^{\frac{1}{2}}<+\infty \bigg\},
\end{split}	
\end{equation}
which is a Banach space equipped with the canonical norm denoted by $\|\cdot\|_{\mathscr{S}_{\lambda,\mu}}$.
\begin{theorem}\label{control1}
Assume that $F\in L^2_{\mathbb{F}}(0,T;L^2(G;\mathbb{C}))$. For any $y_0\in L^2_{\mathcal{F}_{0}}(\Omega;L^2(G;\mathbb{C}))$, there is a control pair $(\hat{h},\hat{H})$ such that the associated solution $\hat{y}$ to the controlled  system $(\ref{equationy3})$ satisfies $\hat y(T)=0$ in $G$, $\mathbb{P}$-a.s. Moreover, one can find two positive constants $\lambda_{0}$ and $\mu_{0}$, such that
\begin{equation}\label{control1est}
\begin{split}
&\EE \int_{Q}\theta^{-2}|\hat y|^2\dxt
+\EE \intt\int_{G_0}\theta^{-2}\lambda^{-3}\mu^{-4}\xi^{-3}|\hat h|^2\dxt
\\
&+\EE \int_{Q}\theta^{-2}\lambda^{-2}\mu^{-2}\xi^{-3}|\hat H|^2\dxt
\\
\leq  &
C\bigg(\EE \int_{G}\lambda^{-2}\mu^{-3}e^{-2\mu(6m+1)}\theta^{-2}(0)|y_{0}|^2\dx
+\|F\|^2_{\mathscr{S}_{\lambda,\mu}}\bigg),
\end{split}
\end{equation}
for all $\lambda\geq\lambda_{0}$ and $\mu\geq\mu_{0}$, where  $C>0$ only depends on $G$ and $G_0$.
\end{theorem}
\begin{proof}[\bf Proof of Theorem \ref{control1}:]
We borrow some ideas from \cite{Liu2014Global}. The main steps are as follows. First, we construct a family of optimal control problems for equation (\ref{equationy3}). Next, we establish a uniform estimate for these optimal solutions. Finally, by taking the limit, one can obtain the estimate (\ref{control1est}) and the desired null controllability result.
We also divide the proof into three parts.

{\noindent\bf Step 1.}
For any $\varepsilon>0$, consider the following weight function
\begin{equation}\nonumber
\gamma_{\varepsilon}(t)=
	\left\{
		\begin{aligned}
	&1+\bigg(1-\frac{4t}{T}\bigg)^{\sigma}, &t\in &[0,T/4),\\
	&1,&t\in &[T/4,T/2+\varepsilon),\\
	&\gamma(t-\varepsilon ),& t\in &[T/2+\varepsilon,T],
	    \end{aligned}
    \right.
\end{equation}
where $\sigma$ is the same as in $(\ref{sigma})$.
Using the new weight function $\gamma_{\varepsilon}(t)$, we set
\begin{equation}\nonumber
	\varphi_{\varepsilon}(t,x):=\gamma_{\varepsilon}(t)\alpha(x),\quad
	\theta_{\varepsilon}:=e^{\lambda\varphi_{\varepsilon} }.
\end{equation}
With this notation, we introduce the functional
\begin{equation}\nonumber
\begin{split}
	J_{\varepsilon}(h,H)
	:=&\frac{1}{2}\EE \int_{Q}\theta_{\varepsilon}^{-2}|y|^2\dxt
	+\frac{1}{2}\EE \intt\int_{G_0}\theta^{-2}\lambda^{-3}\mu^{-4}\xi^{-3}|h|^2\dxt
	\\
	&+\frac{1}{2}\EE \int_{Q}\theta^{-2}\lambda^{-2}\mu^{-2}\xi^{-3}|H|^2\dxt
	+\frac{1}{2\varepsilon}\EE \int_{G}|y(T)|^2\dx
\end{split}	
\end{equation}
and consider the following optimal  control problem:
\begin{equation}\label{61}
\begin{split}
\left\{
\begin{aligned}
	&\min_{(h,H)\in \mathscr{H}}J_{\varepsilon}(h,H)\\
	&\text{subject to equation}\ (\ref{equationy3}),
\end{aligned}
\right.
\end{split}
\end{equation}
where
\begin{equation}\nonumber
\begin{split}
	\mathscr{H}=\bigg\{&(h,H)\in L^2_\mathbb{F}(0,T;L^2(G_0;\mathbb{C}))\times L^2_\mathbb{F}(0,T;L^2(G;\mathbb{C}))\bigg|\\
	&\EE \intt\int_{G_0}\theta^{-2}\lambda^{-3}\mu^{-4}\xi^{-3}|h|^2\dxt<+\infty,\quad
	\EE \int_{Q}\theta^{-2}\lambda^{-2}\mu^{-2}\xi^{-3}|H|^2\dxt<+\infty \bigg\}.
\end{split}	
\end{equation}
Similar to \cite{Li1995Optimal}, it is easy to check that for any $\varepsilon>0$, (\ref{61}) admits a unique optimal pair solution that we denote by $(h_{\varepsilon},H_{\varepsilon})$. Moreover, by the standard variational method (see \cite{Li1995Optimal,Lion1971Optimal}), it follows that
\begin{equation}\label{50}
	h_{\varepsilon}=-\chi _{G_0}\theta^{2}\lambda^{3}\mu^{4}\xi^{3}r_{\varepsilon},\quad
	H_{\varepsilon}=-\theta^{2}\lambda^{2}\mu^{2}\xi^{3}R_{\varepsilon}\quad \text{in}\ Q,\ \mathbb{P}\text{-}a.s.,
\end{equation}
where the pair $(r_{\varepsilon},R_{\varepsilon})$ verifies the backward stochastic equation
\begin{equation}\label{equationz1}
	\left\{
		\begin{aligned}
	&\dd r_{\varepsilon}+(a-\rmm{i} b)\sum_{j,k=1}^{n}(a^{jk}r_{\varepsilon j})_k\dt
	=-\theta_{\varepsilon}^{-2}y_{\varepsilon}\dt+R_{\varepsilon}\dd B(t) &\textup{in}\ &Q,\\
    &r_{\varepsilon}=0 &\textup{on}\ &\Sigma,\\
    &r_{\varepsilon}(T)=\frac{1}{\varepsilon}y_{\varepsilon}(T) &\textup{in}\ &G,
            \end{aligned}
    \right.
\end{equation}
where $y_{\varepsilon}$ is the solution of
(\ref{equationy3}) with the controls $h=h_{\varepsilon}$ and $H=H_{\varepsilon}$.

{\noindent\bf Step 2.} We now establish a uniform estimate for the optimal solutions $\{(y_{\varepsilon},h_{\varepsilon},H_{\varepsilon})\}_{\varepsilon>0}$.
By It\^o's formula, (\ref{equationy3}) and (\ref{equationz1}), it follows that
\begin{equation}\nonumber
\begin{split}
	&\EE\int_{G}y_{\varepsilon}(T)\barr{r_{\varepsilon}}(T)\dx \\
	=& \EE\int_{G}y_{\varepsilon}(0)\barr{r_{\varepsilon}}(0)\dx
	+\EE\int_{Q}\bigg[(a+\rmm{i} b)\sum_{j,k=1}^{n}(a^{jk}y_{\varepsilon j})_k+(F+\chi_{G_0}h_{\varepsilon})\bigg]\barr{r_{\varepsilon}}\dxt \\
	&+\EE\int_{Q}\bigg[-(a+\rmm{i} b)\sum_{j,k=1}^{n}(a^{jk}\barr{r_{\varepsilon }}_{j})_k
	-\theta_{\varepsilon}^{-2}\barr{y_{\varepsilon}}\bigg]y_{\varepsilon}\dxt
	+\EE\int_{Q}H_{\varepsilon}\barr{R_{\varepsilon}}\dxt .
\end{split}
\end{equation}
This, together with (\ref{50}) and the last equality of (\ref{equationz1}), imply that
\begin{equation}\label{51}
\begin{split}
&\EE\intt\int_{G_0}\theta^{2}\lambda^{3}\mu^{4}\xi^{3}|r_{\varepsilon}|^2\dxt
+\EE\int_{Q}\theta^{2}\lambda^{2}\mu^{2}\xi^{3}|R_{\varepsilon}|^2\dxt\\
    &+\EE\int_{Q}\theta_{\varepsilon}^{-2}|y_{\varepsilon}|^2\dxt
	+\frac{1}{\varepsilon}\EE\int_{G}|y_{\varepsilon}(T)|^2\dx \\
	=& \EE\int_{G}y_{0}\barr{r_{\varepsilon}}(0)\dx
	+\EE\int_{Q}F\barr{r_{\varepsilon}}\dxt  .
\end{split}
\end{equation}
Now, we will use the Carleman estimate in Theorem \ref{carle1}. We will apply it to equation (\ref{equationz1}) with $\Xi=-\theta_{\varepsilon}^{-2}y_{\varepsilon}$, $r=r_{\varepsilon}$ and $R=R_{\varepsilon}$. Hence, for any $\lambda\geq\lambda_{0}$ and $\mu\geq\mu_{0}$, we have
\begin{equation}\label{52}
\begin{split}
	&\EE \int_{G}\lambda^2\mu^3e^{2\mu(6m+1)}\theta^2(0)|r_{\varepsilon}(0)|^2\dx
	+\EE \int_{Q}\lambda^3\mu^4\xi^3\theta^2|r_{\varepsilon}|^2\dxt \\
	\leq &
	C\bigg(\EE \int_{0}^{T}\int_{G_0} \lambda^3\mu^4\xi^3\theta^2|r_{\varepsilon}|^2\dxt
	+\EE \int_{Q}\theta^2|\theta_{\varepsilon}^{-2}y_{\varepsilon}|^2\dxt\\
	&+\EE \int_{Q}\lambda^2\mu^2\xi^3\theta^2|R_{\varepsilon}|^2\dxt
	\bigg),
\end{split}
\end{equation}	
In view of (\ref{52}), we use the Young inequality on the right-hand side of (\ref{51}) to obtain
\begin{equation}\label{53}
\begin{split}
&\EE\intt\int_{G_0}\theta^{2}\lambda^{3}\mu^{4}\xi^{3}|r_{\varepsilon}|^2\dxt
+\EE\int_{Q}\theta^{2}\lambda^{2}\mu^{2}\xi^{3}|R_{\varepsilon}|^2\dxt\\
    &+\EE\int_{Q}\theta_{\varepsilon}^{-2}|y_{\varepsilon}|^2\dxt
	+\frac{1}{\varepsilon}\EE\int_{G}|y_{\varepsilon}(T)|^2\dx \\
	\leq &\rho\bigg[\EE \int_{G}\lambda^2\mu^3e^{2\mu(6m+1)}\theta^2(0)|r_{\varepsilon}(0)|^2\dx
	+\EE \int_{Q}\lambda^3\mu^4\xi^3\theta^2|r_{\varepsilon}|^2\dxt \bigg]\\
	&+C_{\rho}\bigg[ \EE \int_{G}\lambda^{-2}\mu^{-3}e^{-2\mu(6m+1)}\theta^{-2}(0)|y_{0}|^2\dx
	+\EE \int_{Q}\lambda^{-3}\mu^{-4}\xi^{-3}\theta^{-2}|F|^2\dxt \bigg].
\end{split}
\end{equation}
for any $\rho>0$. Noting that $\theta^2\theta_{\varepsilon}^{-2}\leq 1$ for any $(t,x)\in Q$ and using (\ref{52}) and (\ref{53}) with sufficiently small $\rho>0$, we obtain that
\begin{equation}\nonumber
\begin{split}
&\EE\intt\int_{G_0}\theta^{2}\lambda^{3}\mu^{4}\xi^{3}|r_{\varepsilon}|^2\dxt
+\EE\int_{Q}\theta^{2}\lambda^{2}\mu^{2}\xi^{3}|R_{\varepsilon}|^2\dxt\\
    &+\EE\int_{Q}\theta_{\varepsilon}^{-2}|y_{\varepsilon}|^2\dxt
	+\frac{1}{\varepsilon}\EE\int_{G}|y_{\varepsilon}(T)|^2\dx \\
	\leq &C\bigg[ \EE \int_{G}\lambda^{-2}\mu^{-3}e^{-2\mu(6m+1)}\theta^{-2}(0)|y_{0}|^2\dx
	+\EE \int_{Q}\lambda^{-3}\mu^{-4}\xi^{-3}\theta^{-2}|F|^2\dxt \bigg].
\end{split}
\end{equation}
Noting that (\ref{50}), we get
\begin{equation}\label{54}
	\begin{split}
&\EE\intt\int_{G_0}\theta^{-2}\lambda^{-3}\mu^{-4}\xi^{-3}|h_{\varepsilon}|^2\dxt
+\EE\int_{Q}\theta^{-2}\lambda^{-2}\mu^{-2}\xi^{-3}|H_{\varepsilon}|^2\dxt\\
    &+\EE\int_{Q}\theta_{\varepsilon}^{-2}|y_{\varepsilon}|^2\dxt
	+\frac{1}{\varepsilon}\EE\int_{G}|y_{\varepsilon}(T)|^2\dx \\
	\leq &C\bigg[ \EE \int_{G}\lambda^{-2}\mu^{-3}e^{-2\mu(6m+1)}\theta^{-2}(0)|y_{0}|^2\dx
	+\EE \int_{Q}\lambda^{-3}\mu^{-4}\xi^{-3}\theta^{-2}|F|^2\dxt \bigg].
\end{split}
\end{equation}

{\noindent\bf Step 3.} By (\ref{54}),  it is easy to check that there exists $(\hat{h},\hat{H},\hat{y})$ such that
\begin{equation}\label{55}
	\left\{
		\begin{aligned}
	&h_{\varepsilon}\rightharpoonup \hat{h} &\textup{weakly in}\ &L^2(\Omega\times(0,T);L^2(G_0;\mathbb{C})),\\
    &H_{\varepsilon}\rightharpoonup \hat{H}  &\textup{weakly in}\ &L^2(\Omega\times(0,T);L^2(G;\mathbb{C})),\\
    &y_{\varepsilon}\rightharpoonup \hat{y}  &\textup{weakly in}\ &L^2(\Omega\times(0,T);L^2(G;\mathbb{C})).
            \end{aligned}
    \right.
\end{equation}
We claim that $\hat{y}$ is the solution to (\ref{equationy3}) associated to $(\hat{h},\hat{H})$. In fact, let us denote by $\tilde{y}\in \mathcal{H}_{T}$ the unique solution to (\ref{equationy3}) with controls $(\hat{h},\hat{H})$. Then, for any $\eta\in L^2_\mathbb{F}(0,T;H_0^1(G;\mathbb{C}))$, consider the following backward stochastic Ginzburg-Landau equation:
\begin{equation}\label{equationz2}
	\left\{
		\begin{aligned}
	&\dd u+(a-\rmm{i} b)\sum_{j,k=1}^{n}(a^{jk}u_{j})_k\dt
	=\eta\dt+U\dd B(t) &\textup{in}\ &Q,\\
    &u=0 &\textup{on}\ &\Sigma,\\
    &u(T)=0 &\textup{in}\ &G.
            \end{aligned}
    \right.
\end{equation}
By (\ref{equationy3}), (\ref{equationz2}) and It\^o's formula, we see that
\begin{equation}\label{56}
\begin{split}
	 -\EE\int_{G}y_0\barr{u}(0)\dx
	=&\EE\int_{Q}F\barr{u}\dxt
	+\EE\intt\int_{G_0}\hat{h}\barr{u}\dxt \\
	&+\EE\int_{Q}\barr\eta \tilde{y}\dxt
	+\EE\int_{Q}\hat{H}\barr{U}\dxt .
\end{split}
\end{equation}
and
\begin{equation}\nonumber
\begin{split}
	 -\EE\int_{G}y_0\barr{u}(0)\dx
	=&\EE\int_{Q}F\barr{u}\dxt
	+\EE\intt\int_{G_0}h_{\varepsilon}\barr{u}\dxt \\
	&+\EE\int_{Q}\barr\eta y_{\varepsilon}\dxt
	+\EE\int_{Q}H_{\varepsilon}\barr{U}\dxt ,
\end{split}
\end{equation}
which, together with (\ref{55}), implys that
\begin{equation}\label{57}
\begin{split}
	 -\EE\int_{G}y_0\barr{u}(0)\dx
	=&\EE\int_{Q}F\barr{u}\dxt
	+\EE\intt\int_{G_0}\hat{h}\barr{u}\dxt \\
	&+\EE\int_{Q}\barr\eta \hat{y}\dxt
	+\EE\int_{Q}\hat{H}\barr{U}\dxt.
\end{split}
\end{equation}
Therefore, by (\ref{56}) and (\ref{57}), we get
$\tilde{y}=\hat{y}$ in $Q$, $\mathbb{P}$-$a.s.$.
Moreover, by (\ref{54}), we get that $\hat{y}(T)=0$ in $G$, $\mathbb{P}$-$a.s.$.
Also, from the weak convergence (\ref{55}), Fatou's lemma and the uniform estimate (\ref{54}), we obtain (\ref{control1est}). Thus, the proof of Theorem \ref{control1} is complete.
\end{proof}
Based on Theorem \ref{control1}, let us prove Theorem \ref{control}.
\begin{proof}[\bf Proof of Theorem \ref{control}:]
According to Theorem \ref{control1}, for any given $F\in L^2_{\mathbb{F}}(0,T;L^2(G;\mathbb{C}))$, we know that there exists a control pair $(h,H)\in L^2_\mathbb{F}(0,T;L^2(G_0;\mathbb{C}))\times L^2_\mathbb{F}(0,T;L^2(G;\mathbb{C}))$, such that the corresponding solution $y\in \mathcal{H}_{T}$ to the controlled system (\ref{equationy3})  satisfies $y(T)=0$ in $G$, $\mathbb{P}$-a.s.
Hence, let us consider a nonlinearity $f$ satisfying assumptions $(H_{4})$, $(H_{5})$ and $(H_{6})$ and define the nonlinear map,
\begin{equation}\nonumber
	\mathscr{E}: F\in {\mathscr{S}_{\lambda,\mu}}\mapsto f(w,t,x,y)\in {\mathscr{S}_{\lambda,\mu}},
\end{equation}
where $y$ is the trajectory of (\ref{equationy3}) associated to the data $y_{0}$ and $F$.
In the following, to simplify the notation, we simply write $f(y)$.

Next, we will show that $\mathscr{E}$ is a contraction mapping from ${\mathscr{S}_{\lambda,\mu}}$ into ${\mathscr{S}_{\lambda,\mu}}$.
First, we check that the mapping $\mathscr{E}$ is well defined. In fact, for any $F\in {\mathscr{S}_{\lambda,\mu}}$, using $(H_{4})$-$(H_{6})$ and $(\ref{control1est})$, we have
\begin{equation}\nonumber
\begin{split}
\|\mathscr{E}F\|_{\mathscr{S}_{\lambda,\mu}}^2
=&\EE \int_{Q}\theta^{-2}\lambda^{-3}\mu^{-4}\xi^{-3}|f(y)|^2\dxt \\
\leq &\kappa^2 \EE \int_{Q}\theta^{-2}\lambda^{-3}\mu^{-4}\xi^{-3}|y|^2\dxt \\
\leq
&\kappa^2\lambda^{-3}\mu^{-4}
\EE\int_{Q}\theta^{-2}|y|^2\dxt
\\
\leq
&C\lambda^{-3}\mu^{-4}
\bigg(\EE \int_{G}\lambda^{-2}\mu^{-3}e^{-2\mu(6m+1)}\theta^{-2}(0)|y_{0}|^2\dx
+\|F\|^2_{\mathscr{S}_{\lambda,\mu}}\bigg)
\\
\leq &\lambda^{-3}\mu^{-4}\bigg(C_{1}\EE\|y_{0}\|^2_{L^2(G)}	+C\|F\|^2_{\mathscr{S}_{\lambda,\mu}}\bigg)
< \infty,
\end{split}	
\end{equation}
for any sufficiently large parameters $\lambda\geq\lambda_{0}$ and $\mu\geq\mu_{0}$, where $C_{1}>0$ depends on $G$,  $G_{0}$, $\lambda$ and $\mu$, and $C > 0$ only depends on $G$ and $G_{0}$. This proves that $\mathscr{E}$ is well defined.

Next, we check that the mapping $\mathscr{E}$ is strictly contractive.
Let $y_{1}$, $y_{2}$ be the solutions of the controlled system (\ref{equationy3}) with respect to the source terms $F_{1}$, $F_{2}\in {\mathscr{S}_{\lambda,\mu}}$, respectively.
Then applying (\ref{control1est}) in Theorem \ref{control1} for the equation associated to  $F=F_{1}-F_{2}$, $y(0)=y_{0}-y_{0}=0$, and using assumption $(H_{6})$, we have
\begin{equation}\nonumber
\begin{split}
	\|\mathscr{E}F_{1}-\mathscr{E}F_{2}\|_{\mathscr{S}_{\lambda,\mu}}
	=&\EE \int_{Q}\theta^{-2}\lambda^{-3}\mu^{-4}\xi^{-3}|f(y_{1})-f(y_{2})|^2\dxt \\
	\leq &\kappa^2 \EE \int_{Q}\theta^{-2}\lambda^{-3}\mu^{-4}\xi^{-3}|y_{1}-y_{2}|^2\dxt \\
	 \leq &C\kappa^2\lambda^{-3}\mu^{-4}\|F_{1}-F_{2}\|^2_{\mathscr{S}_{\lambda,\mu}},
\end{split}	
\end{equation}
for any sufficiently large parameters $\lambda\geq\lambda_{0}$ and $\mu\geq\mu_{0}$, where $C > 0$ only depends on $G$ and $G_{0}$. Thus, if necessary, we can increase the value of $\lambda$ and $\mu$ such that $C\kappa^2\lambda^{-3}\mu^{-4}<1$, which implies that the mapping $\mathscr{E}$ is strictly contractive.

Further, by the Banach fixed point theorem, it follows that $\mathscr{E}$ has a unique fixed point $\tilde F\in {\mathscr{S}_{\lambda,\mu}}$. Moreover, it holds that $\tilde F=f(\omega,t,x,y)$, where $y$ is the solution for the equation (\ref{equationy3}) associated to $F=\tilde F$, i.e. $y$ is the solution to equation (\ref{equationy2}).
Applying Theorem \ref{control1}, we know that there exists a control pair $(h,H)\in L^2_\mathbb{F}(0,T;L^2(G_0;\mathbb{C}))\times L^2_\mathbb{F}(0,T;L^2(G;\mathbb{C}))$, such that $y$ to the controlled system (\ref{equationy2})  satisfies $y(T)=0$ in $G$, $\mathbb{P}$-a.s.
According to Remark \ref{remark1},
the proof of Theorem \ref{control} is complete.
\end{proof}

\section{Controllability of a semilinear backward stochastic complex Ginzburg-Landau equation}
Similar to the proof of Theorem \ref{control},
the key to the proof of Theorem \ref{controlb} is to obtain the controllability results of the following linear backward systems with source terms.
However, unlike the previous section, we will not need to prove the Carleman estimate for the corresponding adjoint system (i.e., the forward equation).
Although this is possible, we can greatly simplify the problem by establishing a Carleman estimate for a random Ginzburg-Landau equation, which can be obtained by the deterministic Carleman estimate.

\subsection{A new Carleman estimate for a deterministic complex Ginzburg-Landau equation}
In this section, we derive a new Carleman estimate for the following deterministic Ginzburg-Landau equation with source term.
\begin{equation}\label{equationqb1}
	\left\{
		\begin{aligned}
	&q_{t}-(a+\rmm{i} b)\sum_{j,k=1}^{n}(a^{jk}q_j)_k=\varpi &\textup{in}\ &Q,\\
    &q=0 &\textup{on}\ &\Sigma,\\
    &q(0)=q_0 &\textup{in}\ &G.
            \end{aligned}
    \right.
\end{equation}
For this, we need to introduce some new auxiliary functions, which are the mirrored version of (\ref{gamma}) and (\ref{99}). In more detail, also set $0<T<1$, we define the function $\tilde\gamma(t)$ as
\begin{equation}\label{gammab}
	\left\{
		\begin{aligned}
	&{\tilde\gamma}(t)=\frac{1}{t^m}, &t\in &[0,T/4),\\
	&{\tilde\gamma}(t)\ \text{is decreasing on}\ [T/4,T/2),\\
	&{\tilde\gamma}(t)=1,&t\in &[T/2,3T/4),\\
	&{\tilde\gamma}(t)=1+\bigg(1-\frac{4(T-t)}{T}\bigg)^{\sigma},& t\in &[3T/4,T],\\
	&{\tilde\gamma}(t)\in C^2([0,T)).
	    \end{aligned}
    \right.
\end{equation}
where $m\geq 1$ and $\sigma\geq 2$ is defined in (\ref{sigma}). Observe that $\tilde\gamma(t)$ is the mirrored version of $\gamma(t)$ in (\ref{sigma}) with respect to $T/2$.
With this new function, we set
\begin{equation}\nonumber
	{\tilde\varphi}(t,x):={\tilde\gamma}(t)\alpha(x),\quad
    \tilde\xi(t,x):=\tilde\gamma(t)e^{\mu(6m+\beta(x))},
\end{equation}
where $\alpha$  is defined in (\ref{99}) and $\mu$ is also a positive parameter with $\mu\geq 1$.
We finally set the weight function
\begin{equation}\label{thetab}
	\tilde\theta :=e^{\tilde\ell},\ \text{where}\ \tilde\ell(t,x):=\lambda\tilde\varphi(t,x).
\end{equation}

We have the Carleman estimate for the deterministic Ginzburg-Landau equation (\ref{equationqb1}).
\begin{theorem}\label{carleb1}
Assume that $\varpi\in L^2(0,T;L^2(G;\mathbb{C}))$,  then there exist $\lambda_{0}>0$ and $\mu_{0}>0$ such that the unique solution $q$ to $(\ref{equationqb1})$ with respect to $q_{0}\in L^{2}(G)$ satisfies
    \begin{equation}\label{carestb1}
    \begin{split}
	&\int_{G}\lambda^2\mu^3e^{2\mu(6m+1)}\tilde\theta^2(T)|z(T)|^2\dx
	+\int_{Q}\lambda\mu^2\tilde\xi\tilde\theta^2|\nabla q|^2\dxt\\
	&+\int_{Q}\lambda^3\mu^4\tilde\xi^3\tilde\theta^2|q|^2\dxt \\
	\leq &
	C\bigg(\int_{0}^{T}\int_{G_0} \lambda^3\mu^4\tilde\xi^3\tilde\theta^2|q|^2\dxt
	+\int_{Q}\tilde\theta^2|\varpi|^2\dxt
	\bigg),
	\end{split}
	\end{equation}	
for all $\lambda\geq\lambda_{0}$ and $\mu\geq\mu_{0}$, where  $C>0$ only depends on $G$ and $G_0$.
\end{theorem}
By using an argument similar to Theorem \ref{carle1}, one can establish the Carleman estimate (\ref{carestb1}) in Theorem \ref{carleb1}. We omit the proof here.

Let us consider the forward complex Ginzburg-Landau equation equation given by
\begin{equation}\label{equationqb2}
	\left\{
		\begin{aligned}
	&\dd q-(a+\rmm{i} b)\sum_{j,k=1}^{n}(a^{jk}q_j)_k\dt=\varpi_1\dt+\varpi_2\dd B(t) &\textup{in}\ &Q,\\
    &q=0 &\textup{on}\ &\Sigma,\\
    &q(0)=q_0 &\textup{in}\ &G.
            \end{aligned}
    \right.
\end{equation}
where $\varpi_{j}\in L^2(0,T;L^2(G;\mathbb{C})) $, $j=1,2$, and $q_{0}\in L^2_{\mathcal{F}_{0}}(\Omega;L^2(G;\mathbb{C}))$.
Notice that when $\varpi_{2}=0$, equation (\ref{equationqb2}) becomes a random Ginzburg-Landau equation. Therefore, using Theorem \ref{carleb1}, we conclude the following global Carleman estimate for random Ginzburg-Landau equations.
\begin{theorem}\label{carleb2}
Assume that $\varpi_{2}\equiv 0$ and $\varpi_{1}\in L^2_{\mathbb{F}}(0,T;L^2(G;\mathbb{C}))$,  then there exist $\lambda_{0}>0$ and $\mu_{0}>0$ such that the unique solution $q\in \mathcal{H}_{T}$ to $(\ref{equationqb2})$ with respect to $q_{0}\in L^{2}_{\mathcal{F}_{0}}(\Omega;L^2(G;\mathbb{C}))$ satisfies
    \begin{equation}\label{carestb2}
    \begin{split}
	&\EE\int_{G}\lambda^2\mu^3e^{2\mu(6m+1)}\tilde\theta^2(T)|q(T)|^2\dx
	+\EE\int_{Q}\lambda\mu^2\tilde\xi\tilde\theta^2|\nabla q|^2\dxt
	\\
	&+\EE\int_{Q}\lambda^3\mu^4\tilde\xi^3\tilde\theta^2|q|^2\dxt \\
	\leq &
	C\bigg(\EE\int_{0}^{T}\int_{G_0} \lambda^3\mu^4\tilde\xi^3\tilde\theta^2|q|^2\dxt
	+\EE\int_{Q}\tilde\theta^2|\varpi|^2\dxt
	\bigg),
	\end{split}
	\end{equation}	
for all $\lambda\geq\lambda_{0}$ and $\mu\geq\mu_{0}$, where  $C>0$ only depends on $G$ and $G_0$.
\end{theorem}

\subsection{The proof of Theorem \ref{controlb}}
In this section, we will give the proof of Theorem \ref{controlb}. Inspired by the duality technique in \cite{Liu2014Global}, by applying Theorem \ref{carleb2}, we first establish a controllability result for the following linear backward stochastic complex Ginzburg-Landau equation.
\begin{equation}\label{equationyb3}
\left\{
		\begin{aligned}
	&\dd y+(a-\rmm{i} b)\sum_{j,k=1}^n (a^{jk}y_j)_k\dt =(F+\chi_{G_0}h)\dt
	+Y\dd B(t) &\textup{in}\ &Q,\\
    &y=0 &\textup{on}\ &\Sigma,\\
    &y(T)=y_T &\textup{in}\ &G,
            \end{aligned}
    \right.
\end{equation}
where $y_{T}\in L^2_{\mathcal{F}_{T}}(\Omega; L^2(G;\mathbb{C}))$ and $F\in L^2_{\mathbb{F}}(0,T;L^2(G;\mathbb{C}))$ are given and $h\in L^2_{\mathbb{F}}(0,T;L^2(G_{0};\\\mathbb{C})$ is a control.
Observe that given the aforementioned regularity on the controls and source term, one can easily show that system (\ref{equationyb3}) admits a unique solution $(y,Y)\in \mathcal{H}_{T}\times L^2_{\mathbb{F}}(0,T;L^2(G;\mathbb{C}))$ (see e.g., \cite[Theorem 4.11]{Lu2021Mathematical}).

We define the space
\begin{equation}\nonumber
\begin{split}
	\mathscr{Q}_{\lambda,\mu}=\bigg\{ F\in L^2_\mathbb{F}(0,T;L^2(G;\mathbb{C})) \bigg|
	\bigg(\EE \int_{Q}\tilde\theta^{-2}\lambda^{-3}\mu^{-4}\tilde\xi^{-3}|F|^2\dxt\bigg)^{\frac{1}{2}}<+\infty \bigg\},
\end{split}	
\end{equation}
which is a Banach space equipped with the canonical norm denoted by $\|\cdot\|_{\mathscr{Q}_{\lambda,\mu}}$.
\begin{theorem}\label{controlb1}
Assume that $F\in L^2_{\mathbb{F}}(0,T;L^2(G;\mathbb{C}))$. For any $y_T\in L^2_{\mathcal{F}_{T}}(\Omega;L^2(G;\mathbb{C}))$, there is a control $\hat{h}$ such that the associated solution $(\hat{y},\hat{Y})\in \mathcal{H}_{T}\times L^2_{\mathbb{F}}(0,T;L^2(G;\mathbb{C}))$ to the controlled  system $(\ref{equationyb3})$ satisfies $\hat y(0)=0$ in $G$, $\mathbb{P}$-a.s. Moreover, one can find two positive constants $\lambda_{0}$ and $\mu_{0}$, such that
\begin{equation}\label{controlb1est}
\begin{split}
&\EE\int_{Q}\tilde\theta^{-2}\lambda^{-2}\mu^{-2}\tilde\xi^{-2}|\hat Y|^2\dxt
+\EE\int_{Q}\tilde\theta^{-2}\lambda^{-2}\mu^{-2}\tilde\xi^{-2}|\nabla \hat y|^2\dxt
\\
&+\EE\int_{Q}\tilde\theta^{-2}|\hat y|^2\dxt
+\EE\intt\int_{G_0}\tilde\theta^{-2}\lambda^{-3}\mu^{-4}\tilde\xi^{-3}|\hat h|^2\dxt
\\
\leq &C\bigg[ \EE\int_{G}\tilde\theta^{-2}(T)\lambda^{-2}\mu^{-2}|y_{T}|^2\dx
+\|F\|^2_{\mathscr{Q}_{\lambda,\mu}} \bigg].
\end{split}
\end{equation}
for all $\lambda\geq\lambda_{0}$ and $\mu\geq\mu_{0}$, where  $C>0$ only depends on $G$ and $G_0$.
\end{theorem}
\begin{proof}[\bf Proof of Theorem \ref{controlb1}]
We divide the proof into three parts.

{\noindent\bf Step 1.}
For any $\varepsilon>0$, consider the following weight function
\begin{equation}\label{70}
\tilde\gamma_{\varepsilon}(t)=
	\left\{
		\begin{aligned}
	&\gamma(t+\varepsilon), &t\in &[0,T/2-\varepsilon),\\
	&1,&t\in &[T/2-\varepsilon,3T/4),\\
	&1+\bigg(1-\frac{4(T-t)}{T}\bigg)^{\sigma},& t\in &[3T/4,T].
	    \end{aligned}
    \right.
\end{equation}
where $\sigma$ is the same as in $(\ref{sigma})$.
In this way, $\tilde\gamma_{\varepsilon}(t)\leq \gamma(t)$ for $t\in [0,T]$.
Using the new weight function $\tilde\gamma_{\varepsilon}(t)$, we set
\begin{equation}\label{71}
	\tilde\varphi_{\varepsilon}(t,x):=\tilde\gamma_{\varepsilon}(t)\alpha(x),\quad
	\tilde\theta_{\varepsilon}:=e^{\lambda\tilde\varphi_{\varepsilon} }.
\end{equation}
With this notation, we introduce the functional
\begin{equation}\nonumber
\begin{split}
	J_{\varepsilon}(h)
	:=&\frac{1}{2}\EE \int_{Q}\tilde\theta_{\varepsilon}^{-2}|y|^2\dxt
	+\frac{1}{2}\EE \intt\int_{G_0}\tilde\theta^{-2}\lambda^{-3}\mu^{-4}\tilde\xi^{-3}|h|^2\dxt
	\\
	&+\frac{1}{2\varepsilon}\EE \int_{G}|y(0)|^2\dx
\end{split}	
\end{equation}
and consider the following optimal  control problem:
\begin{equation}\label{62}
\begin{split}
\left\{
\begin{aligned}
	&\min_{h\in \mathscr{H}}J_{\varepsilon}(h)\\
	&\text{subject to equation}\ (\ref{equationyb3}),
\end{aligned}
\right.
\end{split}
\end{equation}
where
\begin{equation}\nonumber
\begin{split}
	\mathscr{H}=\bigg\{&h\in L^2_\mathbb{F}(0,T;L^2(G_0;\mathbb{C}))\bigg|
	\EE \intt\int_{G_0}\tilde\theta^{-2}\lambda^{-3}\mu^{-4}\tilde\xi^{-3}|h|^2\dxt
	<+\infty \bigg\}.
\end{split}	
\end{equation}
Similar to \cite{Li1995Optimal}, it is easy to check that for any $\varepsilon>0$, (\ref{62}) admits a unique optimal solution that we denote by $h_{\varepsilon}$. Moreover, by the standard variational method (see \cite{Li1995Optimal,Lion1971Optimal}), it follows that
\begin{equation}\label{63}
	h_{\varepsilon}=\chi _{G_0}\tilde\theta^{2}\lambda^{3}\mu^{4}\tilde\xi^{3}p_{\varepsilon}\quad \text{in}\ Q,\ \mathbb{P}\text{-}a.s.,
\end{equation}
where the $r_{\varepsilon}$ verifies the forward random equation
\begin{equation}\label{equationp}
	\left\{
		\begin{aligned}
	&\dd p_{\varepsilon}-(a+\rmm{i} b)\sum_{j,k=1}^{n}(a^{jk}p_{\varepsilon j})_k\dt
	=\tilde\theta_{\varepsilon}^{-2}y_{\varepsilon}\dt &\textup{in}\ &Q,\\
    &p_{\varepsilon}=0 &\textup{on}\ &\Sigma,\\
    &p_{\varepsilon}(0)=\frac{1}{\varepsilon}y_{\varepsilon}(0) &\textup{in}\ &G,
            \end{aligned}
    \right.
\end{equation}
where $y_{\varepsilon}$ is the solution of
(\ref{equationyb3}) with the controls $h=h_{\varepsilon}$.

{\noindent\bf Step 2.} We now establish a uniform estimate for the optimal solutions $\{(y_{\varepsilon},Y_{\varepsilon},h_{\varepsilon})\}_{\varepsilon>0}$.
By It\^o's formula, (\ref{equationyb3}) and (\ref{equationp}), it follows that
\begin{equation}\nonumber
\begin{split}
	&\EE\int_{G}y_{\varepsilon}(T)\barr{p_{\varepsilon}}(T)\dx \\
	=& \EE\int_{G}y_{\varepsilon}(0)\barr{p_{\varepsilon}}(0)\dx
	+\EE\int_{Q}\bigg[-(a-\rmm{i} b)\sum_{j,k=1}^{n}(a^{jk}y_{\varepsilon j})_k+(F+\chi_{G_0}h_{\varepsilon})\bigg]\barr{p_{\varepsilon}}\dxt \\
	&+\EE\int_{Q}\bigg[(a-\rmm{i} b)\sum_{j,k=1}^{n}(a^{jk}\barr{p_{\varepsilon }}_{j})_k
	+\tilde\theta_{\varepsilon}^{-2}\barr{y_{\varepsilon}}\bigg]y_{\varepsilon}\dxt.
\end{split}
\end{equation}
This, together with (\ref{63}) and the last equality of (\ref{equationp}), imply that
\begin{equation}\label{64}
\begin{split}
&\EE\intt\int_{G_0}\tilde\theta^{2}\lambda^{3}\mu^{4}\tilde\xi^{3}|p_{\varepsilon}|^2\dxt
+\EE\int_{Q}\tilde\theta_{\varepsilon}^{-2}|y_{\varepsilon}|^2\dxt
+\frac{1}{\varepsilon}\EE\int_{G}|y_{\varepsilon}(0)|^2\dx \\
=& \EE\int_{G}y_{T}\barr{p_{\varepsilon}}(T)\dx
-\EE\int_{Q}F\barr{p_{\varepsilon}}\dxt  .
\end{split}
\end{equation}
Now, we will use the Carleman estimate (\ref{carestb2}) in Theorem \ref{carleb2} to equation (\ref{equationp}) with $\varpi=\tilde\theta_{\varepsilon}^{-2}y_{\varepsilon}$ and $q=p_{\varepsilon}$. Hence, for any $\lambda\geq\lambda_{0}$ and $\mu\geq\mu_{0}$, we have
\begin{equation}\label{65}
\begin{split}
	&\EE\int_{G}\lambda^2\mu^3e^{2\mu(6m+1)}\tilde\theta^2(T)|p_{\varepsilon}(T)|^2\dx
	+\EE\int_{Q}\lambda\mu^2\tilde\xi\tilde\theta^2|\nabla p_{\varepsilon}|^2\dxt\\
	&+\EE\int_{Q}\lambda^3\mu^4\tilde\xi^3\tilde\theta^2|p_{\varepsilon}|^2\dxt \\
	\leq &
	C\bigg(\EE\int_{0}^{T}\int_{G_0} \lambda^3\mu^4\tilde\xi^3\tilde\theta^2|p_{\varepsilon}|^2\dxt
	+\EE\int_{Q}\tilde\theta^2|\tilde\theta_{\varepsilon}^{-2}y_{\varepsilon}|^2\dxt
	\bigg),
\end{split}
\end{equation}	
In view of (\ref{65}), we use the Young inequality on the right-hand side of (\ref{64}) to obtain
\begin{equation}\label{66}
\begin{split}
&\EE\intt\int_{G_0}\tilde\theta^{2}\lambda^{3}\mu^{4}\tilde\xi^{3}|p_{\varepsilon}|^2\dxt
+\EE\int_{Q}\tilde\theta_{\varepsilon}^{-2}|y_{\varepsilon}|^2\dxt
+\frac{1}{\varepsilon}\EE\int_{G}|y_{\varepsilon}(0)|^2\dx \\
\leq &\rho\bigg[\EE \int_{G}\lambda^2\mu^3e^{2\mu(6m+1)}\tilde\theta^2(T)|p_{\varepsilon}(T)|^2\dx
+\EE \int_{Q}\lambda^3\mu^4\tilde\xi^3\tilde\theta^2|p_{\varepsilon}|^2\dxt \bigg]\\
&+C_{\rho}\bigg[ \EE \int_{G}\lambda^{-2}\mu^{-3}e^{-2\mu(6m+1)}\tilde\theta^{-2}(T)|y_{T}|^2\dx
+\EE \int_{Q}\lambda^{-3}\mu^{-4}\tilde\xi^{-3}\tilde\theta^{-2}|F|^2\dxt \bigg].
\end{split}
\end{equation}
for any $\rho>0$. Noting that $\theta^2\theta_{\varepsilon}^{-2}\leq 1$ for any $(t,x)\in Q$ and using (\ref{65}) and (\ref{66}) with sufficiently small $\rho>0$, we obtain that
\begin{equation}\nonumber
\begin{split}
&\EE\intt\int_{G_0}\tilde\theta^{2}\lambda^{3}\mu^{4}\tilde\xi^{3}|p_{\varepsilon}|^2\dxt
+\EE\int_{Q}\tilde\theta_{\varepsilon}^{-2}|y_{\varepsilon}|^2\dxt
+\frac{1}{\varepsilon}\EE\int_{G}|y_{\varepsilon}(0)|^2\dx \\
\leq &C\bigg[ \EE \int_{G}\lambda^{-2}\mu^{-3}e^{-2\mu(6m+1)}\tilde\theta^{-2}(T)|y_{T}|^2\dx
+\EE \int_{Q}\lambda^{-3}\mu^{-4}\tilde\xi^{-3}\tilde\theta^{-2}|F|^2\dxt \bigg].
\end{split}
\end{equation}
Noting that (\ref{63}), we get
\begin{equation}\label{67}
	\begin{split}
&\EE\intt\int_{G_0}\tilde\theta^{-2}\lambda^{-3}\mu^{-4}\tilde\xi^{-3}|h_{\varepsilon}|^2\dxt
+\EE\int_{Q}\tilde\theta_{\varepsilon}^{-2}|y_{\varepsilon}|^2\dxt
+\frac{1}{\varepsilon}\EE\int_{G}|y_{\varepsilon}(0)|^2\dx \\
\leq &C\bigg[ \EE \int_{G}\lambda^{-2}\mu^{-3}e^{-2\mu(6m+1)}\tilde\theta^{-2}(T)|y_{T}|^2\dx
+\EE \int_{Q}\lambda^{-3}\mu^{-4}\tilde\xi^{-3}\tilde\theta^{-2}|F|^2\dxt \bigg].
\end{split}
\end{equation}

Now we will add a weighted integral of the process $Y$ on the left-hand side of (\ref{67}).
To do that, using It\^o's formula and equation (\ref{equationyb3}) with $h=h_{\varepsilon}$, we have
\begin{equation}\nonumber
\begin{split}
\dd(\tilde\theta_{\varepsilon}^{-2}\lambda^{-2}\tilde\xi^{-2}|y_{\varepsilon}|^2)
=&(\tilde\theta_{\varepsilon}^{-2}\lambda^{-2}\tilde\xi^{-2})_{t}|y_{\varepsilon}|^2\dt
+\tilde\theta_{\varepsilon}^{-2}\lambda^{-2}\tilde\xi^{-2}|Y_{\varepsilon}|^2\dt
\\
&+\tilde\theta_{\varepsilon}^{-2}\lambda^{-2}\tilde\xi^{-2}y_{\varepsilon}\dd \barr{y_{\varepsilon}}
+\tilde\theta_{\varepsilon}^{-2}\lambda^{-2}\tilde\xi^{-2}\barr{y_{\varepsilon}}\dd y_{\varepsilon}.
\end{split}
\end{equation}
Integrating the above equality in $Q$, taking mathematical expectation on both sides, we conclude that
\begin{equation}\label{68}
\begin{split}
&\EE\int_{G}\tilde\theta_{\varepsilon}^{-2}(T)\lambda^{-2}\tilde\xi^{-2}(T)|y_{T}|^2\dx
-\EE\int_{G}\tilde\theta_{\varepsilon}^{-2}(0)\lambda^{-2}\tilde\xi^{-2}(0)|y_{\varepsilon}(0)|^2\dx
\\
=&\EE\int_{Q}(\tilde\theta_{\varepsilon}^{-2}\lambda^{-2}\tilde\xi^{-2})_{t}|y_{\varepsilon}|^2\dxt
+\EE\int_{Q}\tilde\theta_{\varepsilon}^{-2}\lambda^{-2}\tilde\xi^{-2}|Y_{\varepsilon}|^2\dxt
\\
&+\EE\int_{Q}\tilde\theta_{\varepsilon}^{-2}\lambda^{-2}\tilde\xi^{-2}y_{\varepsilon}[-(a+\rmm{i} b)\sum_{j,k=1}^n (a^{jk}\barr{y_{\varepsilon}}_j)_k +(\barr{F}+\chi_{G_0}\barr{h_{\varepsilon}})]\dxt
\\
&+\EE\int_{Q}\tilde\theta_{\varepsilon}^{-2}\lambda^{-2}\tilde\xi^{-2}\barr{y_{\varepsilon}}[-(a-\rmm{i} b)\sum_{j,k=1}^n (a^{jk}{y_{\varepsilon}}_j)_k +(F+\chi_{G_0}h_{\varepsilon})]\dxt.
\end{split}
\end{equation}
Using $\tilde\xi^{-2}(0)=0$, the weight $\tilde\theta_{\varepsilon}^{-1}$ does not blow up at $t=0$, we obtain from (\ref{68}) that
\begin{equation}\label{69}
\begin{split}
&\EE\int_{Q}(\tilde\theta_{\varepsilon}^{-2}\lambda^{-2}\tilde\xi^{-2})_{t}|y_{\varepsilon}|^2\dxt
+\EE\int_{Q}\tilde\theta_{\varepsilon}^{-2}\lambda^{-2}\tilde\xi^{-2}|Y_{\varepsilon}|^2\dxt
\\
&+2a\EE\int_{Q}\tilde\theta_{\varepsilon}^{-2}\lambda^{-2}\tilde\xi^{-2}\sum_{j,k=1}^{n}a^{jk}y_{\varepsilon k}\barr{y_{\varepsilon}}_{j}\dxt
\\
= &\EE\int_{G}\tilde\theta_{\varepsilon}^{-2}(T)\lambda^{-2}\tilde\xi^{-2}(T)|y_{T}|^2\dx
\\
&-2\EE\int_{Q}\rmm{Re}\bigg[(a+\rmm{i}b)y_{\varepsilon}\sum_{j,k=1}^na^{jk}(\tilde\theta_{\varepsilon}^{-2}\lambda^{-2}\tilde\xi^{-2})_{k}\barr{y_{\varepsilon}}_{j}\bigg]\dxt
\\
&-2\EE\int_{Q}\tilde\theta_{\varepsilon}^{-2}\lambda^{-2}\tilde\xi^{-2}\rmm{Re}(y_{\varepsilon}\barr{F})\dxt
-2\EE\intt\int_{G_{0}}\tilde\theta_{\varepsilon}^{-2}\lambda^{-2}\tilde\xi^{-2}\rmm{Re}(y_{\varepsilon}\barr{h_{\varepsilon}})\dxt .
\end{split}	
\end{equation}
Let us analyze some terms in (\ref{69}). First, using $(H_{3})$, we have
\begin{equation}\label{72}
2a\EE\int_{Q}\tilde\theta_{\varepsilon}^{-2}\lambda^{-2}\tilde\xi^{-2}\sum_{j,k=1}^{n}a^{jk}y_{\varepsilon k}\barr{y_{\varepsilon}}_{j}\dxt
\geq
2as_{0}\EE\int_{Q}\tilde\theta_{\varepsilon}^{-2}\lambda^{-2}\tilde\xi^{-2}|\nabla y_{\varepsilon}|^2\dxt.
\end{equation}
Then, by the definition of $\tilde\gamma_{\varepsilon}$, $\tilde\theta_{\varepsilon}$ in (\ref{70}) and (\ref{71}), one can easily obtain that
\begin{equation}\label{73}
\begin{split}
\EE\int_{Q}(\tilde\theta_{\varepsilon}^{-2}\lambda^{-2}\tilde\xi^{-2})_{t}|y_{\varepsilon}|^2\dxt
\geq -C\EE \int_{0}^{3T/4}\int_{G}\tilde\theta_{\varepsilon}^{-2}\lambda^{-1}\mu |y_{\varepsilon}|^2\dxt .
\end{split}	
\end{equation}
Further, using Young inequality and Cauchy inequality, we estimate the last three terms of  (\ref{69}).
For the first one, we have
\begin{equation}\label{74}
\begin{split}
&\bigg|2\EE\int_{Q}\rmm{Re}\bigg[(a+\rmm{i}b)y_{\varepsilon}\sum_{j,k=1}^na^{jk}(\tilde\theta_{\varepsilon}^{-2}\lambda^{-2}\tilde\xi^{-2})_{k}\barr{y_{\varepsilon}}_{j}\bigg]\dxt \bigg|
\\
\leq &\rho \EE\int_{Q}\tilde\theta_{\varepsilon}^{-2}\lambda^{-2}\tilde\xi^{-2}|\nabla y_{\varepsilon}|^2\dxt
+C_{\rho}\EE\int_{Q}\tilde\theta_{\varepsilon}^{-2}\mu^2| y_{\varepsilon}|^2\dxt
\end{split}
\end{equation}
for any $\rho>0$, where we have used $|(\tilde\theta_{\varepsilon}^{-2}\lambda^{-2}\tilde\xi^{-2})_{k}|\leq C\tilde\theta_{\varepsilon}^{-2}\lambda^{-1}\mu\tilde\xi^{-1}$.
For the second one, we obtain
\begin{equation}\label{75}
\begin{split}
&\bigg|2\EE\int_{Q}\tilde\theta_{\varepsilon}^{-2}\lambda^{-2}\tilde\xi^{-2}\rmm{Re}(y_{\varepsilon}\barr{F})\dxt \bigg|
\\
\leq &\EE\int_{Q}\tilde\theta_{\varepsilon}^{-2}\lambda^{-1}\mu^{2}\tilde\xi^{-1}|y_{\varepsilon}|^2\dxt
+\EE\int_{Q}\tilde\theta^{-2}\lambda^{-3}\mu^{-2}\tilde\xi^{-3}|F|^2\dxt,
\end{split}	
\end{equation}
where we have used $\tilde\theta_{\varepsilon}^{-2}\leq \tilde\theta^{-2}$.
For the third one, we get
\begin{equation}\label{76}
\begin{split}
&\bigg|2\EE\intt\int_{G_{0}}\tilde\theta_{\varepsilon}^{-2}\lambda^{-2}\tilde\xi^{-2}\rmm{Re}(y_{\varepsilon}\barr{h_{\varepsilon}})\dxt\bigg|
\\
\leq &\EE\int_{Q}\tilde\theta_{\varepsilon}^{-2}\mu^{2}|y_{\varepsilon}|^2\dxt
+\EE\intt\int_{G_{0}}\tilde\theta_{\varepsilon}^{-2}\lambda^{-4}\mu^{-2}\tilde\xi^{-4}|h_{\varepsilon}|^2\dxt.
\end{split}	
\end{equation}
Combining (\ref{72})-(\ref{76}) with (\ref{69}) and taking $\rho>0$ small enough, we conclude that
\begin{equation}\label{77}
\begin{split}
&\EE\int_{Q}\tilde\theta_{\varepsilon}^{-2}\lambda^{-2}\mu^{-2}\tilde\xi^{-2}|Y_{\varepsilon}|^2\dxt
+\EE\int_{Q}\tilde\theta_{\varepsilon}^{-2}\lambda^{-2}\mu^{-2}\tilde\xi^{-2}|\nabla y_{\varepsilon}|^2\dxt
\\
\leq &C\EE\int_{G}\tilde\theta^{-2}(T)\lambda^{-2}\mu^{-2}|y_{T}|^2\dx
+C\EE\int_{Q}\tilde\theta_{\varepsilon}^{-2}| y_{\varepsilon}|^2\dxt
\\
&+C\EE\int_{Q}\tilde\theta^{-2}\lambda^{-3}\mu^{-4}\tilde\xi^{-3}|F|^2\dxt
+C\EE\intt\int_{G_{0}}\tilde\theta^{-2}\lambda^{-3}\mu^{-4}\tilde\xi^{-3}|h_{\varepsilon}|^2\dxt,
\end{split}	
\end{equation}
where we have used $\tilde\theta_{\varepsilon}^{-2}\leq \tilde\theta^{-2}$, $\tilde\xi^{-2}(T)<1$ and $\lambda^{-1}\tilde\xi^{-1}<1$.
Using (\ref{67}) and (\ref{77}), we have
\begin{equation}\label{78}
\begin{split}
&\EE\int_{Q}\tilde\theta_{\varepsilon}^{-2}\lambda^{-2}\mu^{-2}\tilde\xi^{-2}|Y_{\varepsilon}|^2\dxt
+\EE\int_{Q}\tilde\theta_{\varepsilon}^{-2}\lambda^{-2}\mu^{-2}\tilde\xi^{-2}|\nabla y_{\varepsilon}|^2\dxt
\\
&+\EE\int_{Q}\tilde\theta_{\varepsilon}^{-2}|y_{\varepsilon}|^2\dxt
+\EE\intt\int_{G_0}\tilde\theta^{-2}\lambda^{-3}\mu^{-4}\tilde\xi^{-3}|h_{\varepsilon}|^2\dxt
+\frac{1}{\varepsilon}\EE\int_{G}|y_{\varepsilon}(0)|^2\dx \\
\leq &C\bigg[ \EE\int_{G}\tilde\theta^{-2}(T)\lambda^{-2}\mu^{-2}|y_{T}|^2\dx
+\EE \int_{Q}\lambda^{-3}\mu^{-4}\tilde\xi^{-3}\tilde\theta^{-2}|F|^2\dxt \bigg].
\end{split}
\end{equation}

{\noindent\bf Step 3.} By (\ref{78}),  it is easy to check that there exists $(\hat{h},\hat{y},\hat{Y})$ such that
\begin{equation}\label{79}
	\left\{
		\begin{aligned}
	&h_{\varepsilon}\rightharpoonup \hat{h} &\textup{weakly in}\ &L^2(\Omega\times(0,T);L^2(G_0;\mathbb{C})),\\
    &y_{\varepsilon}\rightharpoonup \hat{y}  &\textup{weakly in}\ &L^2(\Omega\times(0,T);H_0^1(G;\mathbb{C})),\\
    &Y_{\varepsilon}\rightharpoonup \hat{Y}  &\textup{weakly in}\ &L^2(\Omega\times(0,T);L^2(G;\mathbb{C})).
            \end{aligned}
    \right.
\end{equation}
We can check that $(\hat y,\hat Y)$ is the solution to (\ref{equationyb3}) associated to $\hat h$ in exactly the same way as Theorem \ref{control1}.
Moreover, by (\ref{78}), we get that $\hat{y}(0)=0$ in $G$, $\mathbb{P}$-$a.s.$.
Also, from the weak convergence (\ref{79}), Fatou's lemma and the uniform estimate (\ref{78}), we obtain (\ref{controlb1est}). Thus, the proof of Theorem \ref{controlb1} is complete.
\end{proof}

\begin{proof}[\bf Proof of Theorem \ref{controlb}:]
According to Theorem \ref{controlb1}, for any given $F\in L^2_{\mathbb{F}}(0,T;L^2(G;\mathbb{C}))$, we know that there exists a control $h\in L^2_\mathbb{F}(0,T;L^2(G_0;\mathbb{C}))$, such that the corresponding solution $(y,Y)\in \mathcal{H}_{T}\times L^2_\mathbb{F}(0,T;L^2(G;\mathbb{C}))$ to the controlled system (\ref{equationyb3})  satisfies $y(0)=0$ in $G$, $\mathbb{P}$-a.s.
Hence, let us consider a nonlinearity $\Upsilon$ satisfying assumptions $(H_{8})$-$(H_{10})$, and define the nonlinear map,
\begin{equation}\nonumber
	\mathscr{A}: F\in {\mathscr{Q}_{\lambda,\mu}}\mapsto \Upsilon (w,t,x,y,Y)\in {\mathscr{Q}_{\lambda,\mu}},
\end{equation}
where $(y,Y)$ is the trajectory of (\ref{equationyb3}) associated to the data $y_{T}$ and $F$.
In the following, to simplify the notation, we simply write $\Upsilon (y,Y)$.

Next, we will show that $\mathscr{A}$ is a contraction mapping from ${\mathscr{Q}_{\lambda,\mu}}$ into ${\mathscr{Q}_{\lambda,\mu}}$.
First, we check that the mapping $\mathscr{A}$ is well defined. In fact, for any $F\in {\mathscr{Q}_{\lambda,\mu}}$, using $(H_{9})$, $(H_{10})$ and $(\ref{controlb1est})$, we have
\begin{equation}\nonumber
\begin{split}
\|\mathscr{A}F\|_{\mathscr{Q}_{\lambda,\mu}}^2
=&\EE \int_{Q}\theta^{-2}\lambda^{-3}\mu^{-4}\xi^{-3}|\Upsilon (y,Y)|^2\dxt \\
\leq &\kappa_{2}^2 \EE \int_{Q}\theta^{-2}\lambda^{-3}\mu^{-4}\xi^{-3}(|y|^2+|Y|^2)\dxt \\
\leq
&\kappa_{2}^2(\lambda^{-3}\mu^{-4}+\lambda^{-1}\mu^{-2})\bigg(
\EE\int_{Q}\theta^{-2}|y|^2\dxt
+\EE\int_{Q}\theta^{-2}\lambda^{-2}\mu^{-2}\xi^{-2}|Y|^2\dxt\bigg)
\\
\leq
&C\lambda^{-1}\mu^{-2}
\bigg(\EE \int_{G}\lambda^{-2}\mu^{-2}\theta ^{-2}(T)|y_{T}|^2\dx
+\|F\|^2_{\mathscr{Q}_{\lambda,\mu}}\bigg)
\\
\leq &\lambda^{-1}\mu^{-2}\bigg(C_{1}\EE\|y_{T}\|^2_{L^2(G)}	+C\|F\|^2_{\mathscr{Q}_{\lambda,\mu}}\bigg)
< \infty,
\end{split}	
\end{equation}
for any sufficiently large parameters $\lambda\geq\lambda_{0}$ and $\mu\geq\mu_{0}$, where $C_{1}>0$ depends on $G$,  $G_{0}$, $\lambda$ and $\mu$, and $C > 0$ only depends on $G$ and $G_{0}$.
This proves that $\mathscr{A}$ is well defined.

Next, we check that the mapping $\mathscr{A}$ is strictly contractive.
Let $(y_{1},Y_{1})$, $(y_{2},Y_{2})$ be the solutions of the controlled system (\ref{equationyb3}) with respect to the source terms $F_{1}$, $F_{2}\in {\mathscr{Q}_{\lambda,\mu}}$, respectively.
Then applying (\ref{controlb1est}) in Theorem \ref{controlb1} for the equation associated to  $F=F_{1}-F_{2}$, $y(T)=y_{T}-y_{T}=0$, and using assumption $(H_{10})$, we have
\begin{equation}\nonumber
\begin{split}
	\|\mathscr{A}F_{1}-\mathscr{A}F_{2}\|_{\mathscr{Q}_{\lambda,\mu}}
	=&\EE \int_{Q}\theta^{-2}\lambda^{-3}\mu^{-4}\xi^{-3}|\Upsilon (y_{1},Y_{1})- \Upsilon (y_{2},Y_{2})|^2\dxt \\
	\leq &2\kappa_{2}^2 \EE \int_{Q}\theta^{-2}\lambda^{-3}\mu^{-4}\xi^{-3}(|y_{1}-y_{2}|^2+|Y_{1}-Y_{2}|^2)\dxt \\
	 \leq &C\kappa_{2}^2\lambda^{-1}\mu^{-2}\|F_{1}-F_{2}\|^2_{\mathscr{Q}_{\lambda,\mu}},
\end{split}	
\end{equation}
for any sufficiently large parameters $\lambda\geq\lambda_{0}$ and $\mu\geq\mu_{0}$, where $C > 0$ only depends on $G$ and $G_{0}$. Thus, if necessary, we can increase the value of $\lambda$ and $\mu$ such that $C\kappa_{2}^2\lambda^{-1}\mu^{-2}<1$, which implies that the mapping $\mathscr{A}$ is strictly contractive.

Further, by the Banach fixed point theorem, it follows that $\mathscr{A}$ has a unique fixed point $\tilde F\in {\mathscr{Q}_{\lambda,\mu}}$. Moreover, it holds that $\tilde F=\Upsilon (\omega,t,x,y,Y)$, where $(y,Y)$ is the solution for the equation (\ref{equationyb3}) associated to $F=\tilde F$, which means that $(y,Y)$ is the solution to equation (\ref{equationyb1}).
Applying Theorem \ref{controlb1}, we know that there exists a control  $h\in L^2_\mathbb{F}(0,T;L^2(G_0;\mathbb{C}))$, such that $(y,Y)$ to the controlled system (\ref{equationyb1})  satisfies $y(0)=0$ in $G$, $\mathbb{P}$-a.s.
Hence, the proof of Theorem \ref{controlb} is complete.
\end{proof}


\end{document}